\documentclass{aims}
\usepackage{amsmath}
  \usepackage{paralist}
  \usepackage{graphics} 
  \usepackage{epsfig} 
\usepackage{graphicx}  \usepackage{epstopdf}
 \usepackage[colorlinks=true]{hyperref}
 \usepackage{tikz}
\usepackage{pgfplots}
\usetikzlibrary{shapes}
\usetikzlibrary{patterns}
\hypersetup{urlcolor=blue, citecolor=red}

\def\currentvolume{X}
 \def\currentissue{X}
  \def\currentyear{200X}
   \def\currentmonth{XX}
    \def\ppages{X--XX}

\newtheorem{theorem}{Theorem}[section]
\newtheorem{corollary}{Corollary}

\newtheorem{lemma}[theorem]{Lemma}
\newtheorem{proposition}{Proposition}

\theoremstyle{definition}
\newtheorem{definition}[theorem]{Definition}
\newtheorem{remark}{Remark}

\newcommand{\R}{\mathbb{R}}
\newcommand{\N}{\mathbb{N}}
\newcommand{\C}{\mathbb{C}}
\newcommand{\dd}{\mathrm{d}}
\newcommand{\F}{\mathcal{F}}
\newcommand{\eps}{\varepsilon}
\newcommand{\loc}{_{\text{loc}}}

\def\restriction#1#2{\mathchoice
              {\setbox1\hbox{${\displaystyle #1}_{\scriptstyle #2}$}
              \restrictionaux{#1}{#2}}
              {\setbox1\hbox{${\textstyle #1}_{\scriptstyle #2}$}
              \restrictionaux{#1}{#2}}
              {\setbox1\hbox{${\scriptstyle #1}_{\scriptscriptstyle #2}$}
              \restrictionaux{#1}{#2}}
              {\setbox1\hbox{${\scriptscriptstyle #1}_{\scriptscriptstyle #2}$}
              \restrictionaux{#1}{#2}}}
\def\restrictionaux#1#2{{#1\,\smash{\vrule height .8\ht1 depth .85\dp1}}_{\,#2}} 
    \pgfplotsset{
        colormap={parula}{
            rgb255=(53,42,135)
            rgb255=(15,92,221)
            rgb255=(18,125,216)
            rgb255=(7,156,207)
            rgb255=(21,177,180)
            rgb255=(89,189,140)
            rgb255=(165,190,107)
            rgb255=(225,185,82)
            rgb255=(252,206,46)
            rgb255=(249,251,14)
        },
    }

\title[Small defects reconstruction in waveguides] 
      {Small defects reconstruction in waveguides from multifrequency one-side scattering data}

\author[E. Bonnnetier, A. Niclas, L. Seppecher, G. Vial]{}


 \email{eric.bonnetier@univ-grenoble-alpes.fr}
 \email{angele.niclas@ec-lyon.fr}
 \email{laurent.seppecher@ec-lyon.fr}
 \email{gregory.vial@ec-lyon.fr}

\thanks{$^*$ Corresponding author}

\begin{document}
\maketitle

\centerline{\scshape \'Eric Bonnetier}
\medskip
{\footnotesize
 \centerline{Institut Fourier, Université Grenoble Alpes, France}
} 

\medskip

\centerline{\scshape Angèle Niclas$^*$, Laurent Seppecher, Grégory Vial}
\medskip
{\footnotesize
 \centerline{Institut Camille Jordan, \'Ecole Centrale Lyon, France}
}

\bigskip

 \centerline{(Communicated by the associate editor name)}

\begin{abstract}
Localization and reconstruction of small defects in acoustic or electromagnetic waveguides is of
crucial interest in nondestructive evaluation of structures. The aim of this work is to present
a new multi-frequency inversion method to reconstruct small defects in a 2D waveguide. Given
one-side multi-frequency wave field measurements of propagating modes, we use a Born approximation to provide a
$\text{L}^2$-stable reconstruction of three types of defects: a local perturbation inside the waveguide, a
bending of the waveguide, and a localized defect in the geometry of the waveguide. This method
is based on a mode-by-mode spacial Fourier inversion from the available partial data in the Fourier domain.
Indeed, in the available data, some high and low spatial frequency information on the defect are missing. We overcome this issue using both a compact support hypothesis and a minimal smoothness hypothesis on the defects. We also provide a suitable numerical method for efficient 
reconstruction of such defects and we discuss its applications and limits.
\end{abstract}

\section{Introduction}

In this article, we present a method to detect and reconstruct small defects in a waveguide of dimension $2$ from multi-frequency wave field measurements. The measurements are taken on one section of the waveguide, and we assume that only the propagative modes can be detected. Indeed, in most of practical cases, measurements are made far from the defects where the evanescent modes vanish. In a waveguide $\Omega\subset \R^2$, in the time harmonic regime the wave field $u_k$ satisfies the Helmholtz equation 
\begin{equation}
\Delta u_k +k^2(1+q)u_k=-s,
\end{equation}
where $k$ is the frequency, $q$ is a compactly supported bounded perturbation inside the waveguide and the function $s$ is a source of waves.

\begin{figure}[h]
\begin{flushleft} \hspace{2cm}\begin{tikzpicture}
\draw (0,0.5) node{$(1)$};
\draw (1,1) -- (7,1);
\draw (1,0) -- (7,0);
\draw [red] (1.4,-0.2)--(1.4,1.2);
\draw [white,fill=gray!40] (2.3,0.6) circle (0.3);
\draw [<-] [domain=1.2:1.7] [samples=200] plot (\x+0.35,{0.2*sin(15*\x r)+0.6});
\draw [->] [domain=2.5:3.6] [samples=200] plot (\x,{-0.2*sin(15*\x r)+0.65});
\draw [<-] [domain=4.8:5.9] [samples=200] plot (\x-0.3,{0.2*sin(15*\x r)+0.6});
\draw (2.3,0.6) node{$s$};
\draw [red] (1.4,-0.5) node{$\Sigma$};
\draw (3.2,0.25) node{$u_k^\text{inc}$};
\draw (5.2,0.25) node{$u_k^s$};
\draw (6,0.6) [white,fill=gray!20] ellipse (0.4 and 0.3);
\draw (6,0.6) node{$q$};
\end{tikzpicture}

\hspace{2cm}\begin{tikzpicture}
\draw (0,0.5) node{$(2)$};
\draw (1,1) -- (3,1);
\draw (1,0) -- (3,0);
\draw [red] (1.4,-0.2)--(1.4,1.2);
\draw (3,1) arc(90:75:11);
\draw (5.5882,-0.3407)-- ++ (1.4489,-0.3882);
\draw (5.8470,0.6252)-- ++ (1.4489,-0.3882);
\draw (3,0) arc(90:75:10);
\draw [white,fill=gray!40] (2.3,0.6) circle (0.3);
\draw [<-] [domain=1.2:1.7] [samples=200] plot (\x+0.35,{0.2*sin(15*\x r)+0.6});
\draw [->] [domain=2.5:4.4] [samples=200] plot (\x,{-0.2*sin(15*\x r)+0.65});
\draw [<-] [domain=4.8:6] [samples=200] plot (\x,{0.2*sin(15*\x r)+0.2});
\draw (2.3,0.6) node{$s$};
\draw [red] (1.4,-0.5) node{$\Sigma$};
\draw (3.2,0.25) node{$u_k^{\text{inc}}$};
\draw (5.9,-0.2) node{$u_k^s$};
\end{tikzpicture}

\hspace{2cm}\begin{tikzpicture}
\draw (-2,0.5) node{$(3)$};
\draw (-1,1) -- (3,1);
\draw (5,1) -- (5.5,1);
\draw (-1,0) -- (2,0);
\draw (4,0) -- (5.5,0);
\draw [domain=3:5] [samples=200] plot (\x,{-5/16*(\x-3)^2*(\x-5)^2+1});
\draw [domain=2:4] [samples=200] plot (\x,{-4/16*(\x-2)^2*(\x-4)^2});
\draw [white,fill=gray!40] (0.3,0.6) circle (0.3);
\draw [<-] [domain=1.2:1.7] [samples=200] plot (\x+0.35-2,{0.2*sin(15*\x r)+0.6});
\draw [->] [domain=2.5:3.6] [samples=200] plot (\x-2,{-0.2*sin(15*\x r)+0.65});
\draw [red] (-0.6,-0.2)--(-0.6,1.2);
\draw [red] (-0.6,-0.5) node{$\Sigma$};
\draw [<-] [domain=4.6:5.9] [samples=200] plot (\x-2.5,{0.2*sin(15*\x r)+0.4});
\draw (0.3,0.6) node{$s$};
\draw (1.2,0.25) node{$u_k^\text{inc}$};
\draw (3,0) node{$u_k^s$};
\end{tikzpicture}
\end{flushleft}

\caption{\label{debut} Representation of the three types of defects: in $(1)$ a local perturbation $q$, in $(2)$ a bending of the waveguide, in $(3)$ a localized defect in the geometry of $\Omega$. A controlled source $s$ generates a wave field $u^\text{inc}_k$. When it crosses the defect, it generates a scattered wave field $u^s_k$. Both $u^\text{inc}_k$ and $u^s_k$ are measured on the section $\Sigma$.}\end{figure}
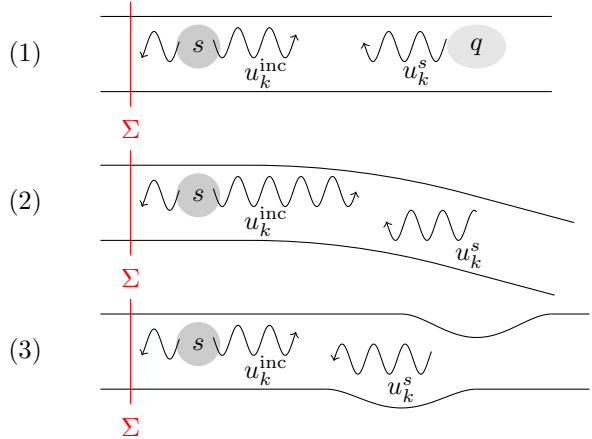

We focus on the inversion of three main types of defects represented in Figure \ref{debut}: a local perturbation of the index $q$, a bend of the waveguide, and a localized defect in the geometry of $\Omega$. The detection of such defects can be used as a non destructive means to monitor pipes, optical fibers, or train rails for instance (see \cite{kharrat1, kharrat2}). A controlled source $s$ generates wave fields in $\Omega$ for some frequencies $k\in K\subset \R^*_+$ and we assume the knowledge of the corresponding measurements $u_k(x,y)$ for every $(x,y)\in \Sigma$ where $\Sigma$ is a fixed section of $\Omega$.

The detection of bends or shape defects in a waveguide is mentioned in the articles \cite{lu1,abra1,abra2}. To solve the forward problem, the authors use a conformal mapping or a local orthogonal transformation to map the geometry to that of a regular waveguide.  This method is very helpful to understand the propagation of waves in irregular waveguides but is not easily adaptable for the inverse problem and for the reconstruction of defects, since the transformation to a regular waveguide is not explicit and proves numerically expensive.

The recovery of inhomogeneities in a waveguide using scattered field data has been extensively studied. In \cite{dediu1}, the authors use a spectral decomposition and assume knowledge of the far-field scattered wave field to reconstruct the inhomogeneities in a 2D waveguide. The authors in \cite{bourgeois1} adapt the Linear Sampling Method \cite{colton2} to waveguides detection of inhomogeneities in 2D or 3D. In \cite{ammari1}, an asymptotic formula of the scattered field is used to localize small inclusions. Periodic waveguides are considered in \cite{bourgeois2}. In all these articles, the frequency in the Helmholtz equation is fixed and it is assumed that incident waves can be sent on every propagative mode in the waveguide. However, as defects may be invisible at some frequencies (as shown in \cite{bonnet1}) the frequency has to be chosen wisely. 

Our work concerns a different approach, also used in \cite{bao1,bao2}, where we assume that data is available for a whole interval of frequencies. This provides additional information that should help not only localize but reconstruct the shape of the defect. The use of multi-frequency data provides uniqueness of the reconstruction (see \cite{acosta1}) and better stability (see \cite{bao3,isakov1,sini1}). In this work we assume that one only send the first propagative mode at different frequencies in the waveguide as an excitation source. This situation seems to correspond to the practice of monitoring pipes in mechanical experiments \cite{kharrat1}. In this study, we assume that the defects are small in amplitude and/or in support in order to approximate the wave field using its Born approximation. This seems to be a reasonable assumption considering the applications that this work intends to address.  This approximation is described in \cite{colton1}, and is also used in \cite{dediu1,ammari1}. Our strategy to study the impact of small geometrical defects is to provide a well suited mapping from the perturbed waveguide to a perfect waveguide that generates some change in the Helmholtz equation itself. Through the reconstruction of these modifications in the equation while assuming a perfect waveguide, it is possible to recover the defects in the geometry.

An important difficulty in detecting inhomogeneities using one sided multi frequencies measurements in a waveguide is that low spacial frequency information carried by vanishing modes about the inhomogeneities may be missing. Indeed, these modes are not measurable in practice due to their exponential decay. 

One of the key results of this article is given by Theorem \ref{therr} that provides conditions to control the error of approximation in the recovery of a function from an incomplete knowledge of its Fourier transform. In this result, we assume that both high frequencies and a reasonable amount of low frequencies are missing. Nevertheless, a stable inversion in $\text{L}^2$ remains possible assuming a reasonable \emph{a priori} knowledge of the smoothness and the support of the unknown perturbation. This result provides a theoretical stability argument that allows us to run a mode-by-mode well-conditioned inversion using a penalized least-square technique. This method is numerically efficient, and can be applied to recover defects of the three different types.

The paper is organized a follows. In section 2, we recall some properties of the forward source problem in a waveguide using the modal decompositions of both the wave field and the source. We then study the inverse source problem with full frequency data and then with partial frequency data.

In section 3, we apply the results to recover all three types of defects that we are interested in: internal inhomogeneities, bending or shape defects. 

In section 4, we present the numerical method used to detect defects and some numerical simulations. To avoid the so called ‘‘inverse crime’’ in the numerical tests, we use two different codes. We use a finite element based solver with PML's \cite{berenger1} to generate the data from a waveguide with defects. Another solver, based on a modal decomposition, allows us to recover the inhomogeneities from the simulated data. Only the second code is used in the inversion procedure.

\section{Forward and inverse source problem in a waveguide}

In this section, we present the tools required to study the forward and inverse source problems in a waveguide. First, we recall some classical results about the forward source problem and modal decomposition. These results can also be found in  \cite{bourgeois1,dediu1}. Next, assuming that the perturbation is small enough, we show existence, uniqueness and stability of a solution to the perturbed forward source problem. Finally, we present an inversion strategy using the measurements of the wave field on a section of the waveguide for full and partial frequency data.

\subsection{Forward source problem in a perfect waveguide}

We consider a 2D infinite perfect waveguide $\Omega=\R\times(0,1)$ in which waves can propagate at frequency $k>0$ according to the homogeneous Helmholtz equation
\begin{equation} \label{bebe} \Delta u_k +k^2u_k=0.\end{equation}
We choose a Neumann condition on the boundary $\partial \Omega$, but this condition can be changed to a Dirichlet or a Robin condition without altering of our results. It is known that the homogeneous Neumann spectral problem for the negative Laplacian on $(0,1)$ has an infinite sequence of eigenvalues $\lambda_n$ for $n\in \N$, and that it is possible to find eigenvectors $\varphi_n$ that form an orthonormal basis of $\text{L}^2(0,1)$. Precisely,
\begin{equation}\lambda_n=n^2\pi^2, \qquad \varphi_n=\left\{\begin{array}{cl} 1 & \text{ if } n=0 ,\\
y\mapsto \sqrt{2}\cos(n\pi y) & \text{ otherwise}.
\end{array}\right. \end{equation}
This basis proves quite helpful in the study of waveguides since every function $f\in \text{L}^2\loc(\Omega)$ can be decomposed as a sum of modes: 
\begin{equation}f(x,y)=\sum_{n\in \N} f_n(x)\varphi_n(y) \quad \text{f. a. e. } (x,y)\in\Omega, \qquad f_n\in \text{L}^2\loc(\R).\end{equation} 

Let $\nu$ be the outward unit normal on $\partial \Omega$. Using this orthonormal basis, the solutions to the homogeneous problem 
\begin{equation}\label{normal}\left\{ \begin{array}{cl} \Delta u_k +k^2 u_k =0 & \text{ in } \Omega, \\ \partial_\nu u_k=0 & \text{ on } \partial\Omega ,\end{array}\right. \end{equation}
are linear combinations of $(x,y)\mapsto\varphi_n(y)e^{\pm ik_n x}$ where $k_n^2=k^2-n^2\pi^2$ and $\text{Re}(k_n)$, $\text{Im} (k_n)\geq 0$. This solution is called the $n$-th mode. 
In the following, we assume that $k_n\neq 0$, meaning that we do not choose a wavelength $k=n\pi$ for $n\in \N$. Two types of modes appear in the decomposition of $u_k$. Propagative modes correspond to $n<k/\pi$ and then $k_n\in \R$, while evanescent modes feature $n> k/\pi$ and $k_n\in i\R$. The amplitude of evanescent modes decays exponentially fast at one end of the waveguide. An extra condition is then needed to ensure the uniqueness of a solution to the Helmholtz problem \eqref{normal}. 

\begin{definition}
A solution $u_k\in \text{H}^2\loc (\Omega)$ of \eqref{bebe} is outgoing if it satisfies the radiation conditions: \begin{equation} \label{sommer}\left| \langle u_k(x,\cdot),\varphi_n\rangle'\frac{x}{|x|}-ik_n\langle u_k(x,\cdot),\varphi_n\rangle \right| \underset{|x|\rightarrow +\infty}{\longrightarrow} 0 \quad \forall n\in \N,\end{equation} 
where $\langle \cdot,\cdot\rangle$ is the inner scalar product in $\text{L}^2(-1,1)$.
\end{definition}

\begin{remark} This condition is an adaptation to our problem of the Sommerfeld condition used in free space. The articles \cite{dediu1,bourgeois1} adopt another radiation condition called Dirichlet to Neumann condition, which is equivalent to our radiation condition when $s$ is compactly supported. \end{remark}

Using the previous conditions, the following proposition holds, the proff of which is given in the Appendix A.
\begin{proposition} \label{source} For every $s\in \text{L}^1(\Omega)\cap \text{L}^2\loc(\Omega)$, the problem 
\begin{equation} \label{directsource} \left\{ \begin{array}{cl} \Delta u_k +k^2 u_k =-s & \text{ in } \Omega, \\ \partial_\nu u_k=0 & \text{ on } \partial\Omega, \\ u_k \text{ is outgoing}, & \end{array}\right. \end{equation}
has a unique solution $u_k\in  \text{H}^2\loc (\Omega)$, which decomposes as 
\begin{equation} \label{inhommod} u_k(x,y)=\sum_{n\in \N} u_{k,n}(x)\varphi_n(y) \quad \text{ where } \quad u_{k,n}(x)=\frac{i}{2k_n}\int_\R s_n(z)e^{ik_n|x-z|}\dd z,\end{equation}
if the decomposition of $s$ is $s(x,y)=\displaystyle\sum_{n\in \N} s_n(x)\varphi_n(y)$.
\end{proposition}

\begin{remark} It is interesting to note that $s$ does not need to have a compact support in this context, as is the case in the free space Helmholtz problem. \end{remark}

Let $\Omega_r:=(-r,r)\times(0,1)$ where $r>0$ denote a restriction of length $2r$ of the waveguide. We assume that every source defined on $\Omega_r$ is extended by $0$ in $\Omega$ and we define the forward Helmholtz source operator $\mathcal{H}_k$ by
\begin{equation}\label{H} \mathcal{H}_k : \begin{array}{rcl}
\text{L}^2(\Omega_r) & \rightarrow & \text{H}^2(\Omega_r) \\
s &\mapsto& \restriction{u_k}{\Omega_r}
\end{array} \qquad \text{ where } u_k \text{ is the solution to \eqref{directsource}} .\end{equation}
The following proposition quantifies the dependence between $u$ and the source $s$. Its proof is given in Appendix B.
\begin{proposition} \label{directh}
The forward Helmholtz source operator $\mathcal{H}_k$ is well defined, continuous and there exists $C>0$ depending only on $k$ and $r$ such that for every $s\in \text{L}^2(\Omega_r)$, 
\begin{equation}\label{contrH}\Vert u_k \Vert_{\text{H}^2(\Omega_r)}\leq C \Vert s\Vert_{\text{L}^2(\Omega_r)}.\end{equation}
\end{proposition}

\begin{remark} \label{distpi}
We notice from the proof that $C$ increases when the distance between $k$ and $\pi\N$ decreases. \end{remark}

In the following, we also need to consider the problem where the source is located on the boundary of the waveguide. Let $\partial\Omega_{\text{top}}=\R\times \{1\}$ and $\partial\Omega_{\text{bot}}=\R\times \{0\}$. Similarly to Proposition \ref{source}, we have 
\begin{proposition} \label{bord}
Let $b_1,b_2 \in \text{L}^1(\R)\cap\text{H}^{1/2}\loc(\R)$. The Helmholtz equation  
\begin{equation}\label{directbord}
\left\{ \begin{array}{cl} \Delta u_k +k^2 u_k =0 & \text{ in } \Omega, \\
\partial_\nu u_k =b_1 & \text{ on } \partial \Omega_{\text{top}}, \\
\partial_\nu u_k =b_2 & \text{ on } \partial \Omega_{\text{bot}}, \\
u_k \text{ is outgoing},
 \end{array} \right. \end{equation}
has a unique solution $u_k \in  \text{H}^2\loc(\Omega)$, which decomposes as 
\begin{equation}u_k(x,y)=\sum_{n\in \N} u_{k,n}(x)\varphi_n(y) \end{equation}
where
\begin{equation} u_{k,n}(x)=\frac{i}{2k_n}\int_\R (b_1(z)\varphi_n(1)+b_2(z)\varphi_n(0))e^{ik_n|x-z|}\dd z.\end{equation}
\end{proposition}

In the restricted guide $\Omega_r$, we assume again that every source defined on $(-r,r)$ is extended by $0$ on $\R$ and we define the forward Helmholtz boundary source operator $\mathcal{G}_k$ by
\begin{equation} \label{G} \mathcal{G}_k : \begin{array}{rcl}
\left(\widetilde{\text{H}}^{1/2}(-r,r)\right)^2 & \rightarrow & \text{H}^2(\Omega_r) \\
(b_1,b_2) &\mapsto& u_k 
\end{array} \qquad \text{ where } u_k \text{ is the solution to \eqref{directbord}},\end{equation}
and $\widetilde{\text{H}}^{1/2}(-r,r)$ is the closure of $\mathcal{D}(-r,r)$, the space of distributions with support in $(-r,r)$, for the $\text{H}^{1/2}(\R)$ norm (see \cite{mclean1} for more details). A result similar to Proposition \ref{directh} holds:
\begin{proposition} \label{directg}
The forward Helmholtz boundary source operator $\mathcal{G}_k$ is well defined, continuous and there exists a constant $D$ depending only on $k$ and $r$ such that for every $b_1,b_2\in \widetilde{\text{H}}^{1/2}(-r,r)$, 
\begin{equation}\label{contrG}\Vert\mathcal{G}(b_1,b_2)\Vert_{\text{H}^2(\Omega_r)}\leq D \left(\Vert b_1 \Vert_{\widetilde{\text{H}}^{1/2}(-r,r)}+\Vert b_2 \Vert_{\widetilde{\text{H}}^{1/2}(-r,r)}\right).\end{equation}
\end{proposition}

\begin{remark}
Combining Propositions \ref{source} and \ref{bord}, we see by linearity that the problem 
\begin{equation}
\left\{ \begin{array}{cl} \Delta u_k +k^2 u_k =-s & \text{ in } \Omega, \\
\partial_\nu u_k =b_1 & \text{ on } \partial \Omega_{\text{top}}, \\
\partial_\nu u_k =b_2 & \text{ on } \partial \Omega_{\text{bot}}, \\
u_k \text{ is outgoing},
 \end{array} \right. \end{equation}
has a unique solution $u_k\in \text{H}^2\loc(\Omega)$. \end{remark}

\subsection{Forward source problem with perturbations \label{secBorn}}

In the following we introduce a theoretical framework for a perturbed Helmholtz problem in a perfect waveguide. Under the Born hypothesis, we prove existence and uniqueness of a solution for the perturbed problem. Then, we provide estimates on the error between the exact solution of the perturbed problem and its Born approximation. 

The perturbed Helmholtz equation takes the form

\begin{equation} \label{perturbed}\left\{ \begin{array}{cl} \Delta w_k +k^2 w_k =-s-\mathcal{S}(w_k) & \text{ in } \Omega, \\ \partial_\nu w_k=b_1+\mathcal{T}_1(w_k) & \text{ on } \partial\Omega_{\text{top}}, \\ \partial_\nu w_k=b_2+\mathcal{T}_2(w_k) & \text{ on } \partial\Omega_{\text{bot}}, \\ w_k \text{ is outgoing}, & \end{array}\right. \end{equation}
where $\mathcal{S},\mathcal{T}_1,\mathcal{T}_2$ are linear operators depending on $w_k$. Moreover, we assume that there exists $r>0$ such that $\text{supp}(\mathcal{S}(w_k))\subset \Omega_r$ and $\text{supp}(\mathcal{T}_1(w_k)),\text{supp}(\mathcal{T}_2(w_k))\subset (-r,r)$ for every $w_k\in \text{H}^2\loc(\Omega)$.

Using the forward Helmholtz source operator $\mathcal{H}_k$ and the forward Helmholtz boundary source operator $\mathcal{G}_k$ defined in \eqref{H} and \eqref{G}, we can rewrite this equation on $\Omega_r$: 
\begin{equation} \label{born2}
w_k=\mathcal{H}_k(s)+\mathcal{G}_k(b_1,b_2)+\mathcal{H}_k(\mathcal{S}(w_k))+\mathcal{G}_k(\mathcal{T}_1(w_k),\mathcal{T}_2(w_k)). \end{equation}

\begin{proposition} \label{borninhomo}
Let $r>0$ such that $\mathcal{S}: \text{H}^2(\Omega_r)\rightarrow \text{L}^2(\Omega_r)$ and $\mathcal{T}_1,\mathcal{T}_2: \text{H}^2(\Omega_r)\rightarrow\widetilde{\text{H}}^{1/2}(-r,r)$. Let $C$ and $D$ be the constants defined in Propositions \ref{directh} and \ref{directg}. Let $s\in \text{L}^2(\Omega_r)$, $b_1,b_2\in \widetilde{\text{H}}^{1/2}(-r,r)$ and assume that 
\begin{equation} \label{hypborn} \begin{split}
\mu:=C \Vert \mathcal{S} \Vert_{\text{H}^2(\Omega_r)\rightarrow \text{L}^2(\Omega_r)}\hspace{7cm}\\+D\left(\Vert \mathcal{T}_1\Vert_{\text{H}^2(\Omega_r)\rightarrow\widetilde{\text{H}}^{1/2}(-r,r)}+\Vert \mathcal{T}_2 \Vert_{\text{H}^2(\Omega_r)\rightarrow\widetilde{\text{H}}^{1/2}(-r,r)}\right)<1.\end{split}
\end{equation}
Then \eqref{born2} has a unique solution $w_k\in \text{H}^2(\Omega_r)$ and 
\begin{equation} \label{serieborn}
w_k=\sum_{m\in \N} \left[\mathcal{H}_k\circ \mathcal{S}+\mathcal{G}_k\circ(\mathcal{T}_1,\mathcal{T}_2)\right]^m\left(\mathcal{H}_k(s)+\mathcal{G}_k(b_1,b_2)\right).
\end{equation}
\end{proposition}

\begin{proof}
If \eqref{hypborn} is satisfied then $\mathcal{H}_k\circ \mathcal{S}+\mathcal{G}_k\circ(\mathcal{T}_1,\mathcal{T}_2)$ is a contraction, and the expression \eqref{serieborn} is the expansion of $w_k$ into a Born series (see for instance \cite{colton1}).\end{proof}
\begin{remark}
In this work, we only consider perturbations which affect the PDE via a linear operator. However, the above Proposition also extends to non linear operators, assuming they are Lipschitz. \end{remark}

To compute numerically $w_k$, we approximate the Born series by its first term. 

\begin{definition} \label{defiborn}
Let $w_k$ be defined by \eqref{serieborn}. We define $v_k$, the Born approximation of $w_k$ by
\begin{equation}
v_k=\mathcal{H}_k(s)+\mathcal{G}_k(b_1,b_2).
\end{equation}
\end{definition}

\begin{proposition} \label{inegaborn}
Assume that $\mu$ satisfies \eqref{hypborn} as in Proposition \ref{borninhomo}. Let $w_k$ be the solution of \eqref{born2} and $v_k$ its Born approximation. Then
\begin{equation} \Vert w_k-v_k\Vert_{\text{H}^2(\Omega_r)}\leq \left(C\Vert s \Vert_{\text{L}^2(\Omega_r)}+D\left(\Vert b_1\Vert_{\widetilde{\text{H}}^{1/2}(-r,r)}+\Vert b_2\Vert_{\widetilde{\text{H}}^{1/2}(-r,r)}\right) \right)\frac{\mu}{1-\mu}. \end{equation} 
\end{proposition}

\begin{proof}
We use the definitions of $w_k$ and $v_k$ and the sum of geometrical series.
\end{proof}

\begin{remark} 
If $f,g_1$ and $g_2$ are small, we have proved that the solution of \eqref{perturbed} is very close to the solution of 
\begin{equation} \left\{ \begin{array}{cl} \Delta v_k +k^2 v_k =-s & \text{ in } \Omega, \\ \partial_\nu v_k=b_1 & \text{ on } \partial\Omega_{\text{top}}, \\ \partial_\nu v_k=b_2 & \text{ on } \partial\Omega_{\text{bot}}, \\ v_k \text{ is outgoing}, & \end{array}\right. \end{equation}
and we have quantified the error made by approximating $w_k$ by $v_k$.  \end{remark}

\subsection{Inverse source problem in a perfect waveguide}

In this section, we consider the inverse problem of reconstructing a real-valued source $s$. The goal is to determine the location of $s$ form measurements made on the section $\{0\}\times (0,1)$ at every frequency $k>0$. 

\begin{center}
\begin{tikzpicture}
\draw (-3,1) -- (3,1);
\draw [dashed] (-5,1) -- (-3,1);
\draw [dashed] (3,1) -- (5,1);
\draw (-3,0) -- (3,0);
\draw [dashed] (-5,0) -- (-3,0);
\draw [dashed] (3,0) -- (5,0);
\draw [white,fill=gray!40] (0,0.5) circle (0.4);
\draw [->] [domain=0.3:1.3] [samples=200] plot (\x,{0.3*sin(10*\x r)+0.5});
\draw [<-] [domain=-1.7:-0.3] [samples=200] plot (\x,{-0.3*sin(10*\x r)+0.5});
\draw (-0.8,0.2) node{$u_k$};
\draw (0,0.6) node{$s$};
\draw (2.5,0.5) node{$\Omega$};
\draw (-2,0.2) node[regular polygon,regular polygon sides=3, fill=black, scale=0.4]{};
\draw (-2,0.5) node[regular polygon,regular polygon sides=3, fill=black, scale=0.4]{};
\draw (-2,0.8) node[regular polygon,regular polygon sides=3, fill=black, scale=0.4]{};
\draw [red] (-2,-0.2) -- (-2,1.2) node[above]{$x=0$};
\end{tikzpicture}
\end{center}

For every $k>0$ and $0<y<1$, $u_k(0,y)$ is measured. Using Proposition \ref{source}, we know that 
\begin{equation} u_k(0,y)=\sum_{n\in \N} u_{k,n}(0)\varphi_n(y) \quad \text{ where } \quad u_{k,n}(0)= \frac{i}{2k_n}\int_\R s_n(z)e^{ik_n|z|}\dd z, \end{equation}
if the decomposition of $s$ is $s(x,y)=\sum_{n\in \N} s_n(x)\varphi_n(y)$. Since 
\begin{equation} u_{k,n}(0)=\int_\R u_k(0,y)\varphi_n(y) \dd y,\end{equation}
 we can theoretically have access to $u_{k,n}(0)$ for every $n\in \N$. However, in real-life experiments, noise is likely to pollute the response of evanescent mode, so we assume that we only have access to $u_{k,n}(0)$ for every $n\in \N$ such that $n<k/\pi$: 
\begin{equation}u_{k,n}(0)=\frac{i}{2k_n}\int_\R s_n(z)e^{ik_n|z|}\dd z \quad \forall k>0, \, \forall n\in \N, \, n<k/\pi .\end{equation}
We notice that this expression depends on $k_n=\sqrt{k^2-n^2\pi^2}$. Since $(k,n)\mapsto (\omega,n):= (\sqrt{k^2-n^2\pi^2},n)$ is one-to-one from $\{(k,n)\in \R_+^*\times \N, \, n<k/\pi\}$ to $\R_+^*\times \N$, the available data is then 
\begin{equation} \label{cv} d_{\omega,n}:=\frac{i}{2\omega}\int_\R s_n(z)e^{i\omega|z|}\dd z \quad \forall n\in \N,\, \forall \omega\in \R^*_+.\end{equation}
This change of variable means that given a mode $n$ and a value $\omega>0$, there exists a frequency $k>0$ such that $n$ is a propagative mode and $k_n=\omega$. In order to remove the absolute value in the expression of the available data, we assume that $\text{supp}(s)\subset (0,+\infty)\times(0,1)$, \textit{i.e.} that the source is located to the right of the section where the measurements are made.
\begin{definition} \label{defigamma} Let $\text{H}$ be the Hilbert space defined by 
\begin{equation} \text{H}:=\left\{\hat{u}:\R^*_+\rightarrow \C \, \vert \, \int_0^{+\infty} \omega^2|\hat{u}(\omega)|^2\dd \omega <+\infty\right\}, \,\Vert \hat{u}\Vert_\text{H}^2=\int_0^{+\infty} \omega^2|\hat{u}(\omega)|^2\dd k.\end{equation}
We denote by $\Gamma$ the forward modal operator and by $F_{\text{source}}$ the forward source operator for problem~\eqref{directsource}.
Then $\Gamma$ and $F_{\text{source}}$ are defined by
\begin{equation} \label{defgamma} \Gamma : \begin{array}{rcl} \text{L}^2(\R_+) & \rightarrow & \text{H} \\ f & \mapsto & \left(\omega\mapsto\displaystyle \frac{i}{2\omega}\int_0^{+\infty} f(z) e^{i\omega z}\dd z \right) \end{array}, \end{equation}
\begin{equation} F_{\text{source}}: \begin{array}{rcl} \text{L}^2(\Omega) & \rightarrow & \ell^2(\text{H}) \\ s & \mapsto & \left(\Gamma(s_n)\right)_{n\in \N} \end{array},\end{equation}
if the decomposition of $s$ is $s(x,y)=\displaystyle\sum_{n\in \N} s_n(x)\varphi_n(y)$.
\end{definition}
We choose the following definition for the Fourier transform: 
$$\F(f)(\omega)=\int_\R f(z)e^{-i\omega z}\dd z. $$
Since $s$ is real-valued, $\Gamma$ is related to the Fourier transform:
\begin{equation*}\F(f)(\omega)=\left\{\begin{array}{cl} \overline{\frac{2\omega}{i} \Gamma(f)(\omega)} & \text{ if } \omega>0 \\ \frac{-2\omega}{i}\Gamma(f)(-\omega) & \text{ if } \omega<0 \end{array} \right. .\end{equation*}
Using the properties of the Fourier transform, we can prove the following Proposition:
\begin{proposition} \label{invgamma} The forward modal operator $\Gamma$ and the forward source operator $F_{\text{source}}$ satisfy the relations 
\begin{equation} \Vert \Gamma(f)\Vert^2_{\text{H}}=\frac{\pi}{4} \Vert f \Vert^2_{\text{L}^2(\R_+)} \quad \forall f\in \text{L}^2(\R_+), \end{equation}
\begin{equation}\Vert F_{\text{source}}(s) \Vert^2_{\ell^2(\text{H})}=\frac{\pi}{4}\Vert s\Vert^2_{\text{L}^2(\Omega)} \quad \forall s\in \text{L}^2(\Omega),
\end{equation}
and their inverse operators are given by 
\begin{equation}\Gamma^{-1}: \begin{array}{rcl} \text{H} & \rightarrow & \text{L}^2(\R) \\ v & \mapsto & \left(x\mapsto \displaystyle\frac{i}{\pi }\displaystyle\int_{0}^{+\infty}\omega\overline{v(\omega)}e^{i\omega x}\dd \omega +\frac{i}{\pi }\int_{-\infty}^0\omega v(-\omega)e^{i\omega x }\dd \omega \right)\end{array},\end{equation}
\begin{equation}F_{\text{source}}^{-1} : \begin{array}{rcl} \ell^2(\text{H}) & \rightarrow & \text{L}^2(\Omega) \\ (v_n)_{n\in \N} & \mapsto & \left((x,y)\mapsto \displaystyle\sum_{n\in \N} \Gamma^{-1}(v_n)(x)\varphi_n(y) \right) \end{array} . \end{equation}
\end{proposition}

We can use the same framework for problem \eqref{directbord} when the source therme is a boundary term. In this case, the measured data is
\begin{equation}u_{k,n}(0)=\frac{i}{2k_n}\int_\R (-b_1(z)\varphi_n(1)+b_2(z)\varphi_n(0))e^{ik_n|z|}\dd z \quad \forall n\in \N,\end{equation}
As $(k,n)\mapsto (\sqrt{k^2-n^2\pi^2},n)$ is one-to-one from $\{(k,n)\in \R_+^*\times \N, \, n<k/\pi\}$ to $\R_+^*\times \N$, we assume that the available data is 
\begin{equation}
d_{\omega,1}=\frac{i}{2\omega}\int_\R(b_1(z)+b_2(z))e^{i\omega|z|}\dd z \qquad \forall \omega\in \R^*_+, \end{equation}
\begin{equation} d_{\omega,2}=\frac{i}{2\omega}\int_\R(-\sqrt{2}b_1(z)+\sqrt{2}b_2(z))e^{i\omega|z|}\dd z \qquad \forall \omega\in \R^*_+. 
\end{equation}
Again, we assume that $\text{supp}(b_1),\text{supp}(b_2)\subset (0,+\infty)$ and with the help of Proposition \ref{bord}, we define the forward operator.

\begin{definition} The forward Helmholtz boundary source operator $F_{\text{bound}}$ for the problem~\eqref{directbord} is defined by
\begin{equation} F_{\text{bound}}: \begin{array}{rcl} \left(\text{H}^{1/2}(\R_+)\right)^2 & \rightarrow & \text{H}\times \text{H}\\ (b_1,b_2) & \mapsto & \left(\begin{array}{c}\omega\mapsto \displaystyle\frac{i}{2\omega}\displaystyle\int_0^{+\infty}(b_1(z)+b_2(z))e^{i\omega z}\dd z \\ \omega\mapsto \displaystyle\frac{i}{\sqrt{2}\omega}\displaystyle\int_0^{+\infty}(b_2(z)-b_1(z))e^{i\omega z}\dd z\end{array}\right)\end{array}. \end{equation}
\end{definition}

This operator is invertible: 
\begin{proposition} \label{fbound} The forward Helmholtz boundary source operator $F_{\text{bound}}$ is invertible: 
\begin{equation} F_{\text{bound}}^{-1} : \begin{array}{rcl} \text{H}\times \text{H} & \rightarrow & \left(\text{H}^{1/2}(\R_+)\right)^2 \\ (v_1,v_2) & \mapsto & \left( \Gamma^{-1}\left(\displaystyle\frac{\sqrt{2}v_1-v_2}{2\sqrt{2}}\right),\Gamma^{-1}\left(\displaystyle\frac{\sqrt{2}v_1+v_2}{2\sqrt{2}}\right) \right) \end{array} . \end{equation}
\end{proposition}

Propositions \ref{invgamma} and \ref{fbound} show that the measurements of the wave on a section of the waveguide for every frequency $k>0$ are sufficient to reconstruct the source. Thus, the inverse operators can be computed explicitly and in a stable way. However, it is unrealistic to measure $u_k$ for every frequency $k>0$ in practice. We address this issue of limited data in the next subsection.

\subsection{Inverse source problem from limited frequency data}

In this section, we assume that the frequency data are only known in a given interval. To reconstruct every $s_n$ in \eqref{directsource}, we need to find a way to reconstruct a function $f$ knowing only the values of its Fourier transform on a given interval. This problem is called Fourier synthesis, and has been studied in \cite{isaev1} for instance. If the given interval has the form $(0,\omega_1)$, some regularity on the function is sufficient to provide a good reconstruction of $f$ and to control the approximation error (see \cite{dym1}). On the other hand, we have to deal in the next section with intervals of the form $(\omega_0,+\infty)$. This case is harder, and it seems difficult to get a good reconstruction of the function $f$. However, if the function $f$ is compactly supported, its Fourier transform is analytic. Thus, the values of $\Gamma(f)(\omega)$ for $\omega$ in a interval $(\omega_0, \omega_1)$ completely determine $\Gamma(f)(\omega)$ for $\omega$ in $(0,+\infty)$. In the following, we address the issue of the stability of this reconstruction. 

We start with a lemma to control the $\text{L}^2$ norm on $(0,\omega_0)$ of an analytic function in therms of its values on $(\omega_0,\omega_0+\sigma)$ where $\omega_0$ and $\sigma$ are positive real numbers. 

\begin{lemma} \label{lemma1}
Let $f$ be a function in $\mathcal{C}^\infty(\R_+) \cap\text{L}^2(\R_+)$ and assume that for every $j\in \N$ and $\omega\in \R_+$, $|f^{(j)}(\omega)|\leq c\frac{r^j}{j^\alpha}\Vert f\Vert_{L^2(\R_+)}$ where $r,\alpha,c \in \R^*_+$. Let $\omega_0,\sigma\in \R^*_+$ and $\eps\in (0,1)$. There exists a constant $\xi$, depending only on $\omega_0,\sigma,r,\alpha,c,\eps$, such that
\begin{equation} \label{majolem} \frac{\Vert f \Vert_{\text{L}^2(0,\omega_0)}}{\Vert f\Vert_{\text{L}^2(\R^+)}}\leq \xi\left(\frac{\Vert f\Vert_{\text{L}^2(\omega_0,\omega_0+\sigma)}}{\Vert f \Vert_{\text{L}^2(\R^+)}}\right)^{1-\eps}.\end{equation}
\end{lemma}

\begin{proof} Let $n\in \N$, we define $\delta_j= c\frac{r^j}{j^\alpha}$ and write the Taylor expansion of $f$ at $\omega_0$ up to order $n$:
\begin{equation*}f(\omega)=\sum_{j=0}^n\frac{(\omega-\omega_0)^j}{j!}f^{(j)}(\omega_0)+R_n(\omega), \end{equation*}
with
\begin{equation*} |R_n(\omega)|\leq \frac{\delta_{n+1}\Vert f\Vert_{L^2(\R_+)}|\omega-\omega_0|^{n+1}}{(n+1)!}.\end{equation*}
We denote by $P_n\in \R_n[X]$ the Taylor polynomial associated with this expansion:  
\begin{equation*}P_n=\sum_{j=0}^n a_j (X-\omega_0)^j := \sum_{j=0}^n\frac{f^{(j)}(\omega_0)}{j!}(X-\omega_0)^j,\end{equation*}
and the operator
\begin{equation*}I_n: \begin{array}{ccc}
\R_n[X]\cap \text{L}^2(\omega_0,\omega_0+\sigma) & \rightarrow & \R_n[X]\cap \text{L}^2(0,\omega_0) \\
P & \mapsto & P
\end{array} , \end{equation*}
endowed with the norm 
\begin{equation*}\Vert I_n\Vert:=\sup_{P\in \R_n[X]}\frac{\Vert P\Vert_{\text{L}^2(0,\omega_0)}}{\Vert P \Vert_{\text{L}^2(\omega_0,\omega_0+\sigma)}}.\end{equation*}
We immediately see that 
\begin{alignat}{2} \Vert f \Vert_{\text{L}^2(0,\omega_0)}&\leq \Vert f-P_n\Vert_{\text{L}^2(0,\omega_0)}+\Vert P_n\Vert_{\text{L}^2(0,\omega_0)} \notag \\  &\leq \Vert R_n\Vert_{\text{L}^2(0,\omega_0)}+ \Vert I_n\Vert \Vert P_n\Vert_{\text{L}^2(\omega_0,\omega_0+\sigma)}\notag \\ &\leq \Vert R_n\Vert_{\text{L}^2(0,\omega_0)}+ \Vert I_n\Vert \Vert R_n\Vert_{\text{L}^2(\omega_0,\omega_0+\sigma)}+\Vert I_n\Vert\Vert f\Vert_{\text{L}^2(\omega_0,\omega_0+\sigma)} \label{i_n}.\end{alignat}
Let us compute $\Vert I_n \Vert$. Let $P=\sum_{j=0}^n a_k(X-\omega_0)^j$ be a polynomial in $\R_n[X]$, then 
\begin{align*} \Vert P\Vert_{\text{L}^2(0,\omega_0)}^2&=\int_{0}^{\omega_0} \sum_{k,p=0}^na_ka_p(\omega-\omega_0)^{k+p}\dd \omega\\ & =\sum_{k,p=0}^n a_ka_p\frac{-(-\omega_0)^{k+p+1}}{k+p+1}=\omega_0W^TH_nW,\end{align*}
where $W:=\left(a_k(-\omega_0)^k\right)_{k=0,\cdots,n}$ and $H_n=\left(\frac{1}{k+p+1}\right)_{p,k=0,\cdots,n}$ is the Hilbert matrix. In the same way, 
\begin{align*}\Vert P\Vert_{\text{L}^2(\omega_0,\omega_0+\sigma)}^2&=\int_{\omega_0}^{\omega_0+\sigma} \sum_{k,p=0}^na_ka_p(\omega-\omega_0)^{k+p}\dd \omega\\ &=\sum_{k,p=0}^n a_ka_p\frac{\sigma^{k+p+1}}{k+p+1}=\sigma V^TH_nV.\end{align*}
where $V:=\left(a_k\sigma^k\right)_{k=0,\cdots,n}$.
Let $\lambda_{\text{min}}$ and $\lambda_{\text{max}}$ be the lowest and greatest eigenvalues of $H_n$. It follows that 
\begin{equation}\label{notopti} \omega_0W^TH_n W\leq \omega_0\Vert W\Vert_2^2\lambda_{\text{max}}, \qquad \sigma V^TH_n V\geq \sigma \Vert V\Vert_2^2 \lambda_{\text{min}}.\end{equation}
Notice that 
\begin{equation*}\Vert W\Vert_2^2\leq \max(1,\omega_0)^{2n}\sum_{k=0}^n |a_k|^2\leq \frac{\max(1,\omega_0)^{2n}}{\min(1,\sigma)^{2n}}\sum_{k=0}^n|a_k|^2\sigma^{2k}\leq \frac{\max(1,\omega_0)^{2n}}{\min(1,\sigma)^{2n}}\Vert V \Vert_2^2.\end{equation*}
Thus, if $\omega_0\leq \sigma$, then $\Vert W\Vert_2^2\leq \Vert V\Vert_2^2$. We follow \cite{todd1} to estimate the condition number of the Hilbert matrix: There exists $c_H>0$ such that $\text{cond}_2(H_n)$ for the euclidean norm satisfies
\begin{equation*}\frac{\lambda_{\text{max}}}{\lambda_{\text{min}}}=\text{cond}_2(H_n)\leq c_H\frac{(1+\sqrt{2})^{4n}}{\sqrt{n}}.\end{equation*}
We conclude that
\begin{equation*}\Vert P\Vert_{\text{L}^2(0,\omega_0)}^2\leq \frac{\omega_0}{\sigma }\left(1_{\omega_0\leq \sigma}+1_{\omega_0>\sigma}\frac{\max(1,\omega_0)^{2n}}{\min(1,\sigma)^{2n}}\right)c_H\frac{(1+\sqrt{2})^{4n}}{\sqrt{n}} \Vert P\Vert_{\text{L}^2(\omega_0,\omega_0+\sigma)}^2.\end{equation*}
We define $C_1:=\sqrt{\frac{c_H\omega_0}{\sigma}}$ and $C_2:=\left(1+\sqrt{2}\right)^2\left(1_{\omega_0\leq\sigma}+1_{\omega_0>\sigma}\frac{\max(1,\omega_0)}{\min(1,\sigma)}\right)$, then 
\begin{equation} \Vert I_n \Vert \leq C_1\frac{C_2^n}{n^{1/4}}.\end{equation}
We next bound $R_n$ in $\text{L}^2(0,\omega_0)$ by 
\begin{align*}\Vert R_n\Vert_{\text{L}^2(0,\omega_0)}&\leq\frac{\delta_{n+1}\Vert f\Vert_{\text{L}^2(\R_+)}}{(n+1)!}\left(\int_{0}^{\omega_0}(\omega_0-\omega)^{2n+2}\dd \omega\right)^{1/2}\\ &\leq \frac{\delta_{n+1}\Vert f\Vert_{\text{L}^2(\R_+)}}{(n+1)!}\left(\frac{\omega_0^{2n+3}}{2n+3}\right)^{1/2},\end{align*}
and in $\text{L}^2(\omega_0,\omega_0+\sigma)$ by
\begin{align*}\Vert R_n\Vert_{\text{L}^2(\omega_0,\omega_0+\sigma)}&\leq\frac{\delta_{n+1}\Vert f\Vert_{\text{L}^2(\R_+)}}{(n+1)!}\left(\int_{\omega_0}^{\omega_0+\sigma}(\omega-\omega_0)^{2n+2}\dd \omega\right)^{1/2}\\ &\leq\frac{\delta_{n+1}\Vert f\Vert_{\text{L}^2(\R_+)}}{(n+1)!}\left(\frac{\sigma^{2n+3}}{2n+3}\right)^{1/2}.\end{align*}
Substituting in \eqref{i_n} we find
\begin{equation}\label{exp1}\begin{split}\Vert f \Vert_{\text{L}^2(0,\omega_0)}\leq\frac{\delta_{n+1}\Vert f\Vert_{\text{L}^2(\R_+)}}{(n+1)!}\frac{\omega_0^{n+3/2}}{\sqrt{2n+3}}+\frac{\delta_{n+1}\Vert f\Vert_{\text{L}^2(\R_+)}}{(n+1)!}\frac{\sigma^{n+3/2}}{\sqrt{2n+3}}\frac{C_1C_2^n}{n^{1/4}}\\+\frac{C_1C_2^n}{n^{1/4}}\Vert f\Vert_{\text{L}^2(\omega_0,\omega_0+\sigma)}.\end{split}\end{equation}
To simplify the notations, we define
\begin{equation*}C_3:=\frac{2}{\sqrt{2}}\max\left(\omega_0^{1/2},C_1\sigma^{1/2}\right)=\max\left(1,\sqrt{c_H}\right)\sqrt{2\omega_0},\quad C_4:=\max(\omega_0,\sigma C_2).\end{equation*}
We notice that
\begin{equation*}C_4=\max\left[\omega_0,\left(1+\sqrt{2}\right)^2\left(1_{\omega_0\leq\sigma}\sigma+1_{\omega_0>\sigma}\frac{\sigma\max(1,\omega_0)}{\min(1,\sigma)}\right)\right]=\sigma c_2. \end{equation*}
The expression \eqref{exp1} can be simplified and 
\begin{equation*}\frac{\Vert f \Vert_{\text{L}^2(0,\omega_0)}}{\Vert f\Vert_{\text{L}^2(0,+\infty)}}\leq C_3\frac{\delta_{n+1}C_4^{n+1}}{(n+1)!\sqrt{n+1}}+C_1C_2^n\frac{\Vert f\Vert_{\text{L}^2(\omega_0,\omega_0+\sigma)}}{\Vert f \Vert_{\text{L}^2(0,+\infty)}} \qquad \forall n\in \N.\end{equation*}
The first term does not depend on $f$, and this expression shows that it is impossible to obtain a Lipschitz estimate. To optimize this estimate, we play on the degree $n$ of the polynomials. Indeed, the first therm on the right hand side may be large for small values of $n$, while the second therm blows up when $n$ is large. We set $Q:=\Vert f\Vert_{\text{L}^2(\omega_0,\omega_0+\sigma)}/\Vert f \Vert_{\text{L}^2(\R^+)}$ and for $\eps\in (0,1)$ we choose the integer \begin{equation*}n=\left\lfloor-\frac{\eps}{\ln(C_2)} \ln(Q)+\frac{\ln(C_5)}{\ln(C_2)}\right\rfloor,\end{equation*}
where $C_5>0$ is a constant to be determined later, and $\lfloor\, \rfloor$ is the floor function. Invoking the Stirling formula $n!\geq \sqrt{2\pi n}\frac{n^n}{e^n}$ and the fact that $\delta_{n+1}=c\frac{r^{n+1}}{(n+1)^\alpha}$, we obtain 
\begin{equation*} C_3\frac{\delta_{n+1}C_4^{n+1}}{(n+1)!\sqrt{n+1}}\leq \frac{C_3c(erC_4)^{n+1}}{(n+1)^{\alpha+1}\sqrt{2\pi}(n+1)^{n+1}}=\frac{C_3c}{\sqrt{2\pi}}\frac{(erC_4)^{n+1}}{(n+1)^{n+\alpha+2}}.\end{equation*}
To simplify the notations, we define
\begin{equation*}\gamma:= erC_4, \quad A:=\frac{\ln(C_5)}{\ln(C_2)}, \quad B:=\frac{\eps}{\ln(C_2)}.\end{equation*}
Using the fact that $A-B\ln(Q)\leq n+1\leq A-B\ln(Q)+1$, we see that 
\begin{equation*}\begin{split}\frac{(ert)^{n+1}}{(n+1)^{n+\alpha+2}}&\leq \exp[ (A-B\ln(Q)+1)\ln(\gamma)\hspace{3cm}\\ &\hspace{3cm} -(A-B\ln(Q)+\alpha+1)\ln(A-B\ln(Q))] \\
&=\gamma^{A+1}Q^{-B\ln(\gamma)+B\ln(A-B\ln(Q))}(A-B\ln(Q))^{-(A+\alpha+1)}.\end{split} \end{equation*}
The exponent of $Q$ is greater that $1-\eps$ provided
\begin{equation*}Q\leq \exp\left(-\frac{1}{B}\left[\exp\left(\frac{1-\eps}{B}+\ln(\gamma)\right)-A\right]\right).\end{equation*}
Since $Q\leq 1$, this condition is satisfied if 
\begin{equation*}A=\exp\left(\frac{1-\eps}{B}+\ln(\gamma)\right)+B\ln(\eta)=erC_4C_2^{\frac{1-\eps}{\eps}},\end{equation*}
which fixes the value of 
\begin{equation*} C_5=C_2^{erC_4C_2^{\frac{1-\eps}{\eps}}}.\end{equation*}
Using the fact that $A-B\ln(Q)\geq erC_4C_2^{\frac{1-\eps}{\eps}}$, it follows that
\begin{equation*}\frac{\Vert f \Vert_{\text{L}^2(0,\omega_0)}}{\Vert f\Vert_{\text{L}^2(\R^+)}}\leq \xi \left(\frac{\Vert f\Vert_{\text{L}^2(\omega_0,\omega_0+\sigma)}}{\Vert f \Vert_{\text{L}^2(\R^+)}}\right)^{1-\eps},\end{equation*}
where 
\begin{equation} \label{formxi} \xi:=\frac{C_3c}{\sqrt{2\pi}}C_2^{-\frac{1-\eps}{\eps}\left(1+\alpha+erC_4C_2^{\frac{1-\eps}{\eps}}\right)}\left(erC_4\right)^{-\alpha} +C_1C_2^{erC_4C_2^{\frac{1-\eps}{\eps}}}. \end{equation}
\end{proof}

\begin{remark} The expression \eqref{formxi} certainly over-estimates the optimal constant in \eqref{majolem}, in particular in view of \eqref{notopti}. \end{remark}

We now consider two functions $f$ and $f_{\text{app}}$ of one variable. The following theorem provides a control over the distance between $f$ and $f_{\text{app}}$ using only the values of their Fourier transforms on the interval $[\omega_0,+\infty)$.

\begin{theorem}[Reconstruction with low frequency gap in the Fourier transform]
\label{ther} Let $f,f_{\text{app}}\in \text{L}^2(-r,r)$ where $r\in \R^*_+$. Let $\omega_0,\sigma\in \R^*_+$. We assume that there exists $M\in \R^*_+$ such that
\begin{equation}\Vert f\Vert_{\text{L}^2(-r,r)}\leq M, \quad \Vert f_{\text{app}} \Vert_{\text{L}^2(-r,r)}\leq M.\end{equation}
For every $0<\eps<1$, there exists $\xi$, depending on $r,\omega_0,\sigma,\eps$, such that 
\begin{equation}\begin{split} \Vert f-f_{\text{app}} \Vert_{\text{L}^2(-r,r)}^2\leq \frac{\left(8\pi M^2\right)^{\eps}}{\pi}\xi^2\Vert \F(f)-\F(f_{\text{app}}) \Vert_{\text{L}^2(\omega_0,\omega_0+\sigma)}^{2-2\eps}\qquad  \\ +\frac{1}{\pi}\Vert \F(f)-\F(f_{\text{app}})\Vert_{\text{L}^2(\omega_0,+\infty)}^2.\end{split} \end{equation}
\end{theorem}

\begin{proof}
We know that 
\begin{equation*}\Vert f-f_{\text{app}}\Vert_{\text{L}^2(-r,r)}^2= \frac{1}{\pi}\Vert \F(f)-\F(f_{\text{app}})\Vert_{\text{L}^2(0,\omega_0)}^2+\frac{1}{\pi}\Vert \F(f)-\F(f_{\text{app}}) \Vert_{\text{L}^2(\omega_0,+\infty)}^2.\end{equation*} 
Since $f$ is compactly supported as a function of $\text{L}^2(\R)$, we know that for every $j\in \N$, $\omega\in \R_+$,
\begin{equation*}\begin{split} \left|\frac{\dd^j}{\dd \omega^j}(\F(f)-\F(f_{\text{app}}))(\omega)\right| &\leq \left(2\int_0^r x^{2j}\dd x\right)^{1/2}\Vert f-f_{\text{app}}\Vert_{L^2(-r,r)} \\ &\leq  \frac{r^{j}\sqrt{r}\Vert\F(f)-\F(f_{\text{app}}) \Vert_{L^2(\R_+)}}{\sqrt{\pi}\sqrt{2j+1}}.\end{split}\end{equation*}
It follows from Lemma \ref{lemma1} that
\begin{equation*}\begin{split} \Vert f -f_{\text{app}} \Vert_{\text{L}^2(-r,r)}^2\leq \frac{1}{\pi}\xi^2\Vert \F(f)-\F(f_{\text{app}})\Vert_{\text{L}^2(\R^+)}^{2\eps}\Vert \F(f)-\F(f_{\text{app}})\Vert_{\text{L}^2(\omega_0,\omega_0+\sigma)}^{2-2\eps}\\+\frac{1}{\pi}\Vert \F(f)-\F(f_{\text{app}})\Vert_{\text{L}^2(\omega_0,+\infty)}^2.\end{split}\end{equation*}
Since $\Vert \F(f)-\F(f_{\text{app}})\Vert_{\text{L}^2(\R^+)}^{2\eps}\leq \left(2\pi\Vert f-f_{\text{app}}\Vert_{\text{L}^2(-r,r)}^2\right)^{\eps}\leq \left(8\pi M^2\right)^{\eps}$, the result follows.
\end{proof}

Next, we generalize Theorem \ref{ther} to the case when we control the Fourier transform of $f$ on a finite interval $[\omega_0, \omega_1]$. 

\begin{theorem}[Reconstruction from a finite interval of the Fourier transform] \label{therr}
Let $f,f_{\text{app}}\in \text{H}^1(-r,r)$ where $r>0$. Let $\omega_0,\omega_1\in \R^*_+,\omega_0<\omega_1$. We assume that there exists $M\in \R^*_+$ such that
\begin{equation}\Vert f\Vert_{\text{H}^1(-r,r)}\leq M, \quad \Vert f_{\text{app}} \Vert_{\text{H}^1(-r,r)}\leq M.\end{equation}
For every $0<\eps<1$, there exists $\xi$, depending on $r,\omega_0,\omega_1,\eps,M$, such that 
\begin{equation}\begin{split} \Vert f-f_{\text{app}} \Vert_{\text{L}^2(-r,r)}^2\leq  \frac{\left(8\pi M^2\right)^{\eps}}{\pi}\xi^2\Vert \F(f)-\F(f_{\text{app}}) \Vert_{\text{L}^2(\omega_0,\omega_1)}^{2-2\eps}\qquad \qquad \\ \qquad +\frac{1}{\pi}\Vert \F(f)-\F(f_{\text{app}})\Vert_{\text{L}^2(\omega_0,\omega_1)}^2+\frac{8}{\omega_1^2}M^2.\end{split}\end{equation}
\end{theorem}

\begin{proof}
We choose $\sigma=\max(1,\omega_1)$ in the Theorem \ref{ther}. Since $f-f_{\text{app}}\in \text{H}^1(-r,r)$, 
\begin{equation*}\begin{split} \Vert \F(f)-\F(f_{\text{app}})\Vert_{\text{L}^2(\omega_1,+\infty)}^2&=\left\Vert \omega\mapsto\frac{ \F(f')(\omega)-\F(f_{\text{app}}')(\omega)}{\omega}\right\Vert_{\text{L}^2(\omega_1,+\infty)}^2 \\ &\leq \frac{2\pi}{\omega_1^2}\Vert f'-f_{\text{app}}'\Vert_{\text{L}^2(-r,r)}^2.\end{split}\end{equation*}

\end{proof}

\begin{remark} \label{remcite}
Using \eqref{formxi}, we notice that $\xi\underset{\omega_0\rightarrow 0}{\longrightarrow} 0$ and that $\frac{8\pi M^2}{\omega_1^2}\underset{\omega_1\rightarrow +\infty}{\longrightarrow} 0$. Moreover, if we define $d=\Gamma(f)$ and $d_{\text{app}}=\Gamma(f_{\text{app}})$ then $2|\omega(d-d_{\text{app}})(\omega)|=|\F(f-f_{\text{app}})(\omega)|$ so \begin{equation}\frac{1}{\pi}\Vert \F(f)-\F(f_{\text{app}})\Vert_{\text{L}^2(\omega_0,\omega_1)}^2\underset{[\omega_0,\omega_1]\rightarrow (0,+\infty)}{\longrightarrow}\frac{4}{\pi}\Vert d-\tilde{d}\Vert_\text{H}^2,\end{equation}
which is consistent with Proposition \ref{invgamma}. \end{remark}

Theorem \ref{therr} provides a theoretical control of the error of the reconstruction between $f$ and $f_{\text{app}}$. However, since $\xi$ can be  very large, such control might not be sufficient to ensure a numerical convergence of $f_{\text{app}}$ to $f$. To illustrate this point, we consider a source $s$ supported on $[1-r,1+r]$. Let $X$ be the discretization of $[1-r,1+r]$ with $N_X$ points. We define $h=2r/(N_X-1)$. Using the fast Fourier transform, we compute the discretization of the Fourier transform $\F(s)$ on a set $K$ of frequencies. We notice that $\F(s)(K)=Ms(X)$ where $M:=h(e^{ixk})_{x\in X, k\in K}$, and that $s(X)=M^{-1}\F(s)(K)$. To simulate the low frequency gap, we truncate $K$ and define $K_{t}=\{k\in K, k>\omega_0\}$ and $M_t=h(e^{ixk})_{x\in X, k\in K_t}$. Then, $s(X)=(M_t^TM_t)^{-1}M_t^T\F(s)(K_t)$. Even if $M_t^TM_t$ is invertible, its condition number strongly depends on $r$ and $\omega_0$ just like the constant $\xi$ in Lemma \ref{lemma1}. Figure \ref{cutcond} illustrates this fact for different values of $\omega_0$ and $r$.

\begin{figure}[!h]
\centering
\scalebox{.55}{
%
%
\definecolor{mycolor1}{rgb}{0.00000,0.44700,0.74100}%
\definecolor{mycolor2}{rgb}{0.85000,0.32500,0.09800}%
\definecolor{mycolor3}{rgb}{0.92900,0.69400,0.12500}%
\definecolor{mycolor4}{rgb}{0.49400,0.18400,0.55600}%
\definecolor{mycolor5}{rgb}{0.46600,0.67400,0.18800}%
\definecolor{mycolor6}{rgb}{0.30100,0.74500,0.93300}%
\begin{tikzpicture}

\begin{semilogyaxis}[%
width=6.028in,
height=4.754in,
at={(0in,0in)},
scale only axis,
xmin=0,
xmax=1,
xlabel style={font=\color{white!15!black}},
xlabel={$r$},
ymin=1,
ymax=100000,
ylabel style={font=\color{white!15!black}},
ylabel={$\text{cond}_2(M_t^TM_t)$},
axis background/.style={fill=white},
grid=minor,
legend style={at={(0.03,0.97)}, anchor=north west, legend cell align=left, align=left, draw=white!15!black}
]
\addplot [color=mycolor1]
  table[row sep=crcr]{%
0.02	1.02020191089183\\
0.04	1.03927737437905\\
0.06	1.05905746304242\\
0.08	1.07957092683371\\
0.1	1.10084786648388\\
0.12	1.1229198032836\\
0.14	1.14581975265303\\
0.16	1.16958230171518\\
0.18	1.19424369109964\\
0.2	1.21984190121607\\
0.22	1.24641674325178\\
0.24	1.27400995516104\\
0.26	1.30266530293069\\
0.28	1.33242868742231\\
0.3	1.36334825710839\\
0.32	1.39547452703962\\
0.34	1.4288605043986\\
0.36	1.46356182101669\\
0.38	1.49963687325214\\
0.4	1.53714696965251\\
0.42	1.57615648684565\\
0.44	1.61673303413313\\
0.46	1.65894762728577\\
0.48	1.70287487206975\\
0.5	1.74859315806374\\
0.52	1.79618486336076\\
0.54	1.84573657078118\\
0.56	1.89733929626202\\
0.58	1.95108873012697\\
0.6	2.0070854919807\\
0.62	2.06543540001841\\
0.64	2.12624975558454\\
0.66	2.18964564386746\\
0.68	2.25574625166597\\
0.7	2.32468120322327\\
0.72	2.39658691517887\\
0.74	2.47160697175307\\
0.76	2.54989252135016\\
0.78	2.63160269582595\\
0.8	2.71690505375254\\
0.82	2.80597604908494\\
0.84	2.89900152672379\\
0.86	2.99617724655017\\
0.88	3.09770943761971\\
0.9	3.20381538428472\\
0.92	3.3147240461385\\
0.94	3.4306767137796\\
0.96	3.55192770252216\\
0.98	3.67874508630326\\
1	3.81141147418241\\
};
\addlegendentry{$\pi$}

\addplot [color=mycolor2]
  table[row sep=crcr]{%
0.02	1.0436109960383\\
0.04	1.08657560685493\\
0.06	1.13298833707292\\
0.08	1.18315675999763\\
0.1	1.23741992591571\\
0.12	1.29615187436101\\
0.14	1.35976555495103\\
0.16	1.42871720558488\\
0.18	1.50351124264693\\
0.2	1.58470572440666\\
0.22	1.6729184561323\\
0.24	1.76883381363153\\
0.26	1.87321037111026\\
0.28	1.98688942951265\\
0.3	2.11080455301226\\
0.32	2.24599223421686\\
0.34	2.39360382310001\\
0.36	2.55491887087738\\
0.38	2.73136005822651\\
0.4	2.92450989765898\\
0.42	3.13612942276435\\
0.44	3.36817910279742\\
0.46	3.62284225001564\\
0.48	3.90255121970677\\
0.5	4.21001673944538\\
0.52	4.54826074529175\\
0.54	4.92065314897439\\
0.56	5.33095301225983\\
0.58	5.78335466342544\\
0.6	6.28253935685384\\
0.62	6.83373315121904\\
0.64	7.44277176554259\\
0.66	8.11617326681831\\
0.68	8.86121954920521\\
0.7	9.68604768455022\\
0.72	10.5997523588581\\
0.74	11.6125007612549\\
0.76	12.7356614631637\\
0.78	13.9819490181729\\
0.8	15.3655862303952\\
0.82	16.902486283904\\
0.84	18.6104572017979\\
0.86	20.5094314144245\\
0.88	22.6217235670816\\
0.9	24.9723200927718\\
0.92	27.5892045216708\\
0.94	30.5037230018918\\
0.96	33.7509950736897\\
0.98	37.3703753795428\\
1	41.4059727153466\\
};
\addlegendentry{$2\pi$}

\addplot [color=mycolor3]
  table[row sep=crcr]{%
0.02	1.06808196470298\\
0.04	1.1380865736308\\
0.06	1.21692671370377\\
0.08	1.3058529894343\\
0.1	1.40631896919322\\
0.12	1.52001693704282\\
0.14	1.64892018988787\\
0.16	1.79533310177826\\
0.18	1.96195040477881\\
0.2	2.15192740548909\\
0.22	2.36896317602994\\
0.24	2.61739913769335\\
0.26	2.9023359060272\\
0.28	3.22977180176945\\
0.3	3.60676706946622\\
0.32	4.04163860486318\\
0.34	4.54419089729894\\
0.36	5.12598997313328\\
0.38	5.80068841509679\\
0.4	6.5844110713598\\
0.42	7.49621290594034\\
0.44	8.55862263708133\\
0.46	9.79828843180572\\
0.48	11.2467450558234\\
0.5	12.941325617238\\
0.52	14.926245508295\\
0.54	17.2538914831804\\
0.56	19.9863551811385\\
0.58	23.1972580158985\\
0.6	26.9739234479459\\
0.62	31.419963527445\\
0.64	36.6583595922062\\
0.66	42.8351325467947\\
0.68	50.1237167383794\\
0.7	58.7301736857931\\
0.72	68.8994085336967\\
0.74	80.9225839619728\\
0.76	95.1459644229645\\
0.78	111.98146925494\\
0.8	131.919267930715\\
0.82	155.542816245284\\
0.84	183.546810789646\\
0.86	216.758633196457\\
0.88	256.163968500236\\
0.9	302.937417277004\\
0.92	358.479083529357\\
0.94	424.458314977125\\
0.96	502.866006022599\\
0.98	596.077154018097\\
1	706.925696005842\\
};
\addlegendentry{$3\pi$}

\addplot [color=mycolor4]
  table[row sep=crcr]{%
0.02	1.09366801856695\\
0.04	1.19422768217086\\
0.06	1.31239284776073\\
0.08	1.45165108487025\\
0.1	1.61628157205771\\
0.12	1.81154935506132\\
0.14	2.04394900415433\\
0.16	2.32151041044932\\
0.18	2.65418273251596\\
0.2	3.05431661799551\\
0.22	3.5372700047243\\
0.24	4.12216934000577\\
0.26	4.83286631284447\\
0.28	5.69914064054398\\
0.3	6.75821268258041\\
0.32	8.0566464234456\\
0.34	9.6527446227421\\
0.36	11.6195648758291\\
0.38	14.0487194798534\\
0.4	17.0551652807649\\
0.42	20.7832445264351\\
0.44	25.4143072692288\\
0.46	31.1763339907356\\
0.48	38.3560888781496\\
0.5	47.3144759495989\\
0.52	58.5059501279303\\
0.54	72.5030637365701\\
0.56	90.027518894209\\
0.58	111.989464664284\\
0.6	139.537245893461\\
0.62	174.120405606836\\
0.64	217.569499208491\\
0.66	272.197240647702\\
0.68	340.926724337519\\
0.7	427.454023592071\\
0.72	536.454448036171\\
0.74	673.844265378163\\
0.76	847.112905828753\\
0.78	1065.74476008108\\
0.8	1341.75489645555\\
0.82	1690.36966896173\\
0.84	2130.89166088514\\
0.86	2687.7992130305\\
0.88	3392.14456734269\\
0.9	4283.33224024697\\
0.92	5411.38168071329\\
0.94	6839.80691493838\\
0.96	8649.28245911446\\
0.98	10942.311502962\\
1	13849.1720575104\\
};
\addlegendentry{$4\pi$}

\addplot [color=mycolor5]
  table[row sep=crcr]{%
0.02	1.12042542471836\\
0.04	1.2554638205837\\
0.06	1.42117211210537\\
0.08	1.62549936348632\\
0.1	1.87875277824383\\
0.12	2.19435115425502\\
0.14	2.58982699479906\\
0.16	3.08816014554565\\
0.18	3.71955336948826\\
0.2	4.52379708273583\\
0.22	5.55341980020074\\
0.24	6.87788707940232\\
0.26	8.58920082939359\\
0.28	10.8093707234847\\
0.3	13.7003907738162\\
0.32	17.4775711785474\\
0.34	22.4273675373885\\
0.36	28.9312423164004\\
0.38	37.4976219457238\\
0.4	48.8047243274838\\
0.42	63.7579895602097\\
0.44	83.5671374858034\\
0.46	109.849615625191\\
0.48	144.76954748409\\
0.5	191.224456713382\\
0.52	253.096314732309\\
0.54	335.589227121442\\
0.56	445.683863703312\\
0.58	592.749261234138\\
0.6	789.366850811038\\
0.62	1052.44079131033\\
0.64	1404.69469616158\\
0.66	1876.69002365949\\
0.68	2509.54901498045\\
0.7	3358.62952266317\\
0.72	4498.48636332825\\
0.74	6029.57207802211\\
0.76	8087.29022346384\\
0.78	10854.2315279702\\
0.8	14576.7177765832\\
0.82	19587.1777912133\\
0.84	26334.4219281034\\
0.86	35424.6172260597\\
0.88	47676.7642073391\\
0.9	64197.8329139293\\
0.92	86484.5587361171\\
0.94	116561.403143865\\
0.96	157167.589127348\\
0.98	212010.750919201\\
1	286111.035303598\\
};
\addlegendentry{$5\pi$}

\addplot [color=mycolor6]
  table[row sep=crcr]{%
0.02	1.14841370943942\\
0.04	1.32231324594948\\
0.06	1.54536520545807\\
0.08	1.83355121606324\\
0.1	2.20879166380095\\
0.12	2.70131345028602\\
0.14	3.35298943243953\\
0.16	4.22205110284328\\
0.18	5.38974224287903\\
0.2	6.96971767277754\\
0.22	9.12132841969764\\
0.24	12.0684163042297\\
0.26	16.1259292252437\\
0.28	21.7376516550406\\
0.3	29.5297490742365\\
0.32	40.3868306559418\\
0.34	55.560100140482\\
0.36	76.8212617869551\\
0.38	106.681709381346\\
0.4	148.704916224731\\
0.42	207.951960473462\\
0.44	291.617339279104\\
0.46	409.936910998968\\
0.48	577.485211175188\\
0.5	815.030193234106\\
0.52	1152.18637775318\\
0.54	1631.21213671559\\
0.56	2312.4473281924\\
0.58	3282.10380117958\\
0.6	4663.43231959039\\
0.62	6632.7368697237\\
0.64	9442.35117491292\\
0.66	13453.61915653\\
0.68	19184.256001995\\
0.7	27376.3897531179\\
0.72	39094.3552446809\\
0.74	55865.3087778468\\
0.76	79881.4962877643\\
0.78	114291.32497445\\
0.8	163618.393361981\\
0.82	234364.96894532\\
0.84	335881.440723134\\
0.86	481619.452986805\\
0.88	690938.722584039\\
0.9	991713.161472582\\
0.92	1424091.31020186\\
0.94	2045924.37049545\\
0.96	2940604.24090681\\
0.98	4228385.74507488\\
1	6082747.82279915\\
};
\addlegendentry{$6\pi$}

\end{semilogyaxis}

\end{tikzpicture}%
}
\caption{\label{cutcond} Condition number of $M_t^TM_t$ for different sizes of support and values of $\omega_0$. Here, $X$ is the discretization of $[1-r,1+r]$ with $500r+1$ points. The $x$-axis represents the evolution of $r$, and the $y$-axis $\text{cond}_2(M_t^TM_t)$. Each curve corresponds to value of $\omega_0$ as indicated in the left rectangle.}
\end{figure}
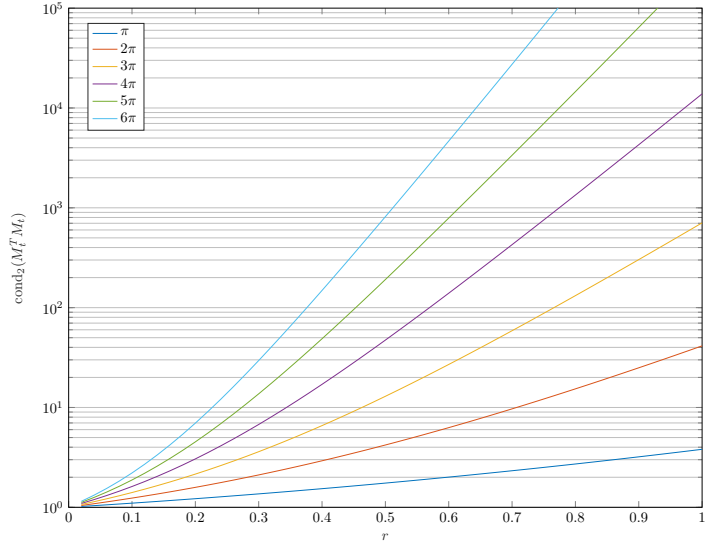

To conclude, if we only have access to perturbed Fourier transform data on a given interval of frequencies, we can build an approximation of $f$ provided $f$ is compactly supported and an \emph{a priori} bound on the norm of $f$ is known. However, depending of $\omega_0$ and $r$, the error between $f$ and its approximation can be large. We can reduce it by increasing $\omega_1$ and by diminishing $\omega_0$ and $r$.

\section{Application to the identification of shape defects, bending or inhomogeneity}

We propose a method to identify shape defects or bends in a waveguide, which is almost identical to our method of source detection. We first map the deformed waveguide to a regular waveguide, and then use the source inverse method discussed in the first section to reconstruct the parameters that characterize the defect.

\subsection{Transformation of the deformed waveguide}

Let $\phi_0$ and $\phi_1$ in $\mathcal{C}^1(\R)$. We consider a deformed waveguide
\begin{equation}
\widetilde{\Omega}=\bigcup_{x\in \R}(\phi_0(x),\phi_1(x))=\{\phi_0(x)<y<\phi_1(x), x\in \R\}. 
\end{equation}  
A wave $\tilde{u}$ in $\widetilde{\Omega}$ satisfies the equation 
\begin{equation} \label{avantvar} \left\{ \begin{array}{cl} \Delta \tilde{u} +k^2\tilde{u} =-\tilde{s} &  \text{ in } \widetilde{\Omega}, \\
\partial_\nu \tilde{u}=\tilde{b}_1 & \text{ on } \partial\widetilde{\Omega}_{\text{top}}, \\
\partial_\nu \tilde{u}=\tilde{b}_2 &  \text{ on }\partial\widetilde{\Omega}_{\text{bot}}, \\
\tilde{u} \text{ is outgoing}\textcolor{red}{,}
\end{array}\right. \end{equation}
where $\tilde{s}\in \text{L}^2\loc(\widetilde{\Omega})$, $\tilde{b_1}\in \text{H}^{1/2}\loc(\partial\widetilde{\Omega}_{\text{top}}),$ and $\tilde{b_2}\in \text{H}^{1/2}\loc(\partial\widetilde{\Omega}_{\text{bot}})$.
To use the tools developed in the previous section, we map $\widetilde{\Omega}$ to the regular waveguide $\Omega=(0,1)\times \R$. Let $\phi$ be a one-to-one function that maps $\Omega$ into $\widetilde{\Omega}$. Such a function exists and can even be assumed to be conformal (see for instance \cite{abra1}). We define $u=\tilde{u}\circ \phi$ the wave in the regular guide, $J\phi$ the Jacobien matrix of $\phi$, $\tau=|\text{det} (J\phi)|$, $t_1=\vert \nabla \phi_0\vert$ , and $t_2=\vert \nabla \phi_1\vert $. The variational formulation of \eqref{avantvar} shows that for every $\tilde{v}\in \text{H}^1(\widetilde{\Omega})$, 

\begin{equation*}\int_{\widetilde{\Omega}} \nabla \tilde{u} \cdot \nabla \tilde{v}-k^2\int_{\widetilde{\Omega}}\tilde{u} \tilde{v}=\int_{\widetilde{\Omega}} \tilde{s}\tilde{v}+\int_{\partial \widetilde{\Omega}_{\text{top}}} \tilde{b}_1\tilde{v} +\int_{\partial \widetilde{\Omega}_{\text{bot}}} \tilde{b}_2\tilde{v}, \end{equation*}
or equivalently,
\begin{equation}\begin{split} \int_{\Omega} (\nabla\tilde{u}\circ \phi)\cdot (\nabla \tilde{v}\circ \phi)\tau -k^2\int_{\Omega} (\tilde{u}\circ \phi) (\tilde{v}\circ \phi)\tau =\int_{\Omega} (\tilde{s}\circ \phi)(\tilde{v}\circ \phi)\tau\\+ \int_{\R} (\tilde{b}_1\circ \phi_1 \,t_1+\tilde{b}_2\circ \phi_0 \,t_2)\, \tilde{v}\circ \phi. \end{split}\end{equation}
Using the fact that $\nabla u =J\phi^T\nabla \tilde{u}\circ \phi$, we set $s=\tilde{s}\circ \phi$, $b_1=\tilde{b}_1\circ \phi_1$, $b_2=\tilde{b}_2\circ\phi_0$, and obtain that for every $v\in \text{H}^1(\Omega)$,

\begin{equation}\int_{\Omega}S\nabla u \cdot \nabla v -k^2\int_{\Omega} u\,v\,\tau =\int_{\Omega} s\, v\, \tau +\int_{\partial \Omega_{\text{top}}} b_1\, v \,t_1+\int_{\partial \Omega_{\text{bot}}} b_2\, v \,t_2.  \end{equation}
where $S=J\phi^{-1}\left(J\phi^{-1}\right)^{T}\tau$, which yields the equation satisfied by $u$: 
\begin{equation}\label{droit}\left\{\begin{array}{cl}
\nabla\cdot(S\nabla u)+k^2\tau u=-\tau s & \text{ in } \Omega, \\ S\nabla u \cdot \nu =b_1t_1 & \text{ on }\partial \Omega_{\text{top}}, \\ S\nabla u \cdot \nu =b_2t_2 & \text{ on }  \partial \Omega_{\text{bot}}. \end{array} \right.\end{equation}
We write $S=I_2+M$ and $\tau=1+\eps$, where $M$ and $\eps$ are expected to be small if the deformation is small. The above partial differential equation becomes 

\begin{equation}\label{umod}\left\{\begin{array}{cl}
\Delta u +k^2u=-\tau s -\nabla\cdot(M\nabla u)-k^2\eps u & \text{ in } \Omega, \\\nabla u \cdot \nu =b_1t_1-M\nabla u\cdot \nu &  \text{ on }\partial \Omega_{\text{top}}, \\ \nabla u \cdot \nu = b_2t_2-M\nabla u\cdot \nu & \text{ on } \partial \Omega_{\text{bot}}.\end{array} \right.\end{equation}
For $r>0$, we set  

\begin{equation}\Sigma : \begin{array}{rcl} \text{H}^2(\Omega_r)& \rightarrow & \text{L}^2(\Omega_r)\\
u & \mapsto & \nabla\cdot (M\nabla u )+k^2\eps u \end{array}, \qquad \Pi : \begin{array}{rcl} \text{H}^2(\Omega_r)& \rightarrow & \widetilde{\text{H}}^{1/2}(-r,r)\\
u & \mapsto & M\nabla u\cdot \nu \end{array}. \end{equation}

The next Proposition follows from the definitions of $\Sigma$ and $\Pi$ and the dependence between $\Vert M\Vert_{\mathcal{C}^1(\Omega_r)}$, $\eps$ and
 $\phi$.
 
\begin{proposition} \label{pcontrP}
The operator $\Sigma$ and $\Pi$ are continuous if $M\in \mathcal{C}^1(\Omega_r)$. In addition, there exists constants $A(\phi),B(\phi)$ depending only on $k$ and $r$ such that
\begin{equation} \label{contrP}\Vert \Sigma(u) \Vert_{\text{L}^2(\Omega_r)}\leq A(\phi)\Vert u \Vert_{\text{H}^2(\Omega_r)}, \qquad \Vert \Pi(u) \Vert_{\text{H}^{1/2}(\Omega_r)}\leq B(\phi)\Vert u \Vert_{\text{H}^2(\Omega_r)}.\end{equation}\end{proposition}

Recalling Section \ref{secBorn} and Definition \ref{defiborn}, we define the Born approximation $v$ of $u$ by
\begin{equation}\left\{\begin{array}{cl}
\Delta v+k^2v=-\tau s & \text{ in } \Omega, \\ \nabla v \cdot \nu =b_1t_1 &  \text{ on } \partial \Omega_{\text{top}}, \\ \nabla v \cdot \nu =b_2t_2 &  \text{ on } \partial \Omega_{\text{bot}}, \\ v \text{ is outgoing.}\end{array} \right.\end{equation}
Proposition \ref{borninhomo} and \ref{inegaborn} yield the following: 
\begin{proposition} \label{born} Let $C$ and $D$ be the constants defined in Propositions \ref{directh} and \ref{directg}, and $A(\phi), B(\phi)$ defined in Proposition \ref{pcontrP}.  If $CA(\phi)+2DB(\phi)<1$ then \eqref{umod} has a unique solution $u$ and  
\begin{equation} \begin{split} \Vert u-v\Vert_{\text{H}^2(\Omega_r)}\leq \frac{CA(\phi)+2DB(\phi)}{1-CA(\phi)+2DB(\phi)}[C\Vert \tau s \Vert_{\text{L}^2(\Omega_r)}\qquad \qquad \\ \qquad \qquad +D\left(\Vert b_2t_2 \Vert_{\widetilde{\text{H}}^{1/2}(-r,r)}+\Vert b_1t_1 \Vert_{\widetilde{\text{H}}^{1/2}(-r,r)}\right)].\end{split} \end{equation}\end{proposition}

The Born approximation leads to a problem of source inversion similar to that of section 2. Using the results proved in this section, we recover $\tau s$, $b_1t_1$ and $b_2t_2$. In the following, we study how to characterize a defect by recovering one of those functions. In the case of a bend, one can fix $b_1=b_2=0$ and reduce the inversion to the sole recovery of $\tau s$. In the case of a bump, $s=0$ and the problem reduces to the reconstruction of $b_1t_1$ and $b_2t_2$. 

\subsection{Detection of bends}
We first consider bends which are parallel portions of circular arcs, whose geometry is determined by the center and the arc-length of these arcs,
or equivalently by the distance $x_c$ where the guide starts bending, the angle $\theta$ and the radius of curvature $r$ (see Figure \ref{bend}).

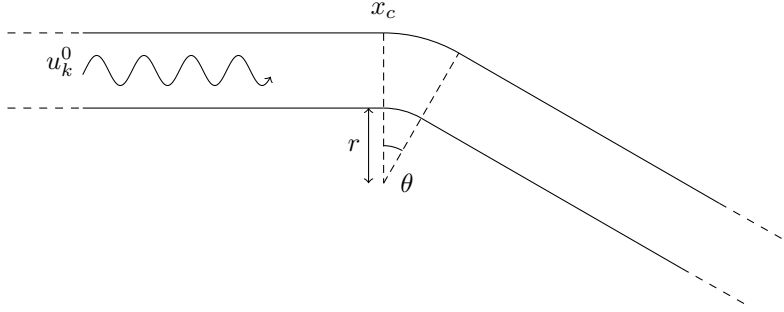
\begin{figure}[h]
\begin{center}
\begin{tikzpicture}
\draw (0,1) -- (4,1);
\draw [dashed] (-1,1) -- (0,1);
\draw [dashed] (-1,0) -- (0,0);
\draw (0,0) -- (4,0);
\draw (4,1) arc(90:60:2);
\draw (4.5,-0.134)-- ++ (3.4641,-2);
\draw (5,0.7321)-- ++ (3.4641,-2);
\draw [dashed] (8.4641,-1.2679)-- ++ (0.886,-0.5);
\draw [dashed] (7.9641,-2.134)-- ++ (0.886,-0.5);
\draw (4,0) arc(90:60:1);
\draw (-0.3,0.3) node[above]{$u_k^0$};
\draw [->] [domain=0:2.5] [samples=200] plot (\x,{0.2*sin(10*(\x+0.6) r)+0.5});
\draw  [densely dashed] (4,1) -- (4,-1);
\draw [densely dashed] (4,-1) -- (5,0.7321);
\draw (4,-0.5) arc(90:60:0.5);
\draw (4.1,-1) node[right]{$\theta$};
\draw [<->] (3.8,0) -- (3.8, -1);
\draw (3.4,-0.5) node[right]{$r$};
\draw (4,1.3) node{$x_c$};
\end{tikzpicture} \caption{\label{bend} Representation of a bend in a waveguide.} \end{center} \end{figure}

More precisely, we define the mapping $\phi$ from $\Omega$ to $\widetilde{\Omega}$ as follow: 
\begin{itemize} 
\item If $x\leq x_c$, $\phi(x,y)= (x,y)$.
\item If $x\in (x_c,x_c+\theta(r+1))$, then 
$$\phi(x,y)=\left(x_c+(r+y)\sin \left(\frac{x-x_c}{r+1}\right),-r+(r+y)\cos\left(\frac{x-x_c}{r+1}\right)\right). $$
\item If $x\geq x_c+\theta(r+1)$ then
\begin{equation*}\begin{split} \phi(x,y)=(x_c+(r+y)\sin\theta+(x-x_c-\theta(r+1))\cos\theta,-r+(r+y)\cos\theta\\-(x-x_c-\theta(r+1))\sin\theta).\end{split}\end{equation*}
\end{itemize}
The matrix $J\phi$ is orthogonal if $x\not\in (x_c,x_c+\theta(r+1)) $ and so $\tau=1$ in this range. If $x\in  (x_c,x_c+\theta(r+1))$, then 
\begin{equation*}J\phi(x,y)=\left(\begin{array}{cc} \frac{r+y}{r+1}\cos\left(\frac{x-x_c}{r+1}\right) & \sin \left(\frac{x-x_c}{r+1}\right) \\
-\frac{r+y}{r+1}\sin \left(\frac{x-x_c}{r+1}\right) & \cos\left(\frac{x-x_c}{r+1}\right)\end{array}\right),\quad \tau=\frac{r+y}{r+1},\end{equation*}
\begin{equation} S=J\phi^{-1}\left(J\phi^{-1}\right)^T\tau=\left(\begin{array}{cc} \frac{r+1}{r+y} &0 \\ 0 & \frac{r+y}{r+1} \end{array}\right).\end{equation}
Moreover, $t_1=1$ for every $x\in \R$, $t_2=1$ if $x\not\in (x_c,x_c+\theta(r+1))$ and $t_2=\frac{r}{r+1}$ otherwise. 

We assume along this section that the bend is located to the right of the section $\{0\}\times (0,1)$. We introduce a source $\tilde{s}_k=-2ik\delta_0(x)$, and we notice that $s_k=\tilde{s}_k\circ\phi=\tilde{s}_k$. In the absence of defect, the wave field generated by this source would be $u_k^{\text{inc}}:=e^{ik|x|}$. Let $u^s_k$ be the scattered wave field defined by $u^s_k:=u_k-u_k^{\text{inc}}$. Using \eqref{umod}, we notice that $u^s_k$ satisfies the equation
\begin{equation}\left\{\begin{array}{cl}
\nabla(S\nabla u^s_k)+k^2\tau u^s_k=-\tau s_k-\nabla(S\nabla u_k^{\text{inc}})-k^2du_k^{\text{inc}} & \text{ in }\Omega, \\ 
S\nabla u^s_k\cdot \nu =-S\nabla u_k^{\text{inc}}\cdot \nu & \text{ on }\partial\Omega, \\
u^s_k \text{ is outgoing.}
\end{array}\right.  \end{equation}
The fact that $S\nabla u_k^{\text{inc}}\cdot \nu=0$, and 
\begin{equation} -\tau s_k-\nabla(S\nabla u_k^{\text{inc}})-k^2du_k^{\text{inc}}=-\textbf{1}_{x\in[x_c,x_c+\theta(r+1)]}k^2e^{ikx}h_r(y), \end{equation}
 with $h_r(y)=(y-1)\left(\frac{1}{r+y}+\frac{1}{r+1}\right)$
leads to the equation 
\begin{equation} \label{bendd}\left\{\begin{array}{cl}
\nabla(S\nabla u^s_k)+k^2\tau u^s_k=-\textbf{1}_{x\in[x_c,x_c+\theta(r+1)]}k^2e^{ikx}h_r(y) & \text{ in } \Omega, \\ 
S\nabla u^s_k\cdot \nu =0 & \text{ on } \partial\Omega, \\
u^s_k \text{ is outgoing. }
\end{array}\right.  \end{equation}
Under the assumptions of Proposition \ref{born}, $u^s_k$ is close to the solution $v_k$ of 
\begin{equation}\left\{\begin{array}{cl}
\Delta v_k+k^2v_k=-\textbf{1}_{x\in[x_c,x_c+\theta(r+1)]}k^2e^{ikx}h_r(y) &  \text{ in } \Omega, \\ 
\nabla v_k\cdot \nu =0 & \text{ on } \partial\Omega, \\ v_k \text{ is outgoing.}
\end{array}\right.  \end{equation}
 
The measurements consist in the first mode $v_{k,0}$ of $v_k$ for every frequency $k\in (0,k_{\text{max}})$ where $k_{\text{max}}\in \R^*_+$ is given. To simplify the source in \eqref{bendd}, we define 
\begin{equation} \label{hry} f=\textbf{1}_{x\in[x_c,x_c+\theta(r+1)]}\int_0^1h_r(t)\dd t.\end{equation}
Proposition \ref{source} yields
\begin{equation} \label{gammabend}
v_{k,0}(0)=\frac{i}{2k}\int_0^{+\infty} k^2 f(y) e^{2iky}\dd y=2k^2 \Gamma(f)(2k) \qquad \forall k\in(0,k_{\text{max}}),
\end{equation}
which shows that we have access to $\Gamma(f)(k)$ for all $k\in (0,2k_{\text{max}})$. We denote by $d=\Gamma(f)(k)$ the data and by $d_{\text{app}}$ the perturbed data. We use the method described in section 2 to reconstruct an approximation $f_{\text{app}}$ of $f$. The error is controlled by the following: 

\begin{proposition} \label{contr}
Let $f$ and $f_{\text{app}}$ be two indicator functions supported in $(-a,a)$ where $a>0$. We assume that the size of the supports of $f$ and $f_{\text{app}}$ is greater than $\delta$. Let $k_{\text{max}}\in R^*_+$, $d(k)=\Gamma(f)(k)$ and $d_{\text{app}}(k)=\Gamma(f_{\text{app}})(k)$ defined for $k\in (0,2k_{\text{max}})$. Let $c(k)=(\int_k^{+\infty}\text{sinc}^2(x)\dd x)^{1/2}$. Then there exists a constant $M\in \R^*_+$ such that
\begin{equation} \Vert f-f_{\text{app}} \Vert_{\text{L}^2(-a,a)}^2\leq \frac{4}{\pi}\Vert d-d_{\text{app}}\Vert_\text{H}^2+Mc(\delta \,k_{\text{max}}). \end{equation}
\end{proposition}

\begin{proof} We notice that $|\F(f)(k)|=2|k\Gamma(f)(k)|$ and we use the fact that the Fourier transform of a indicator function is a $\text{sinc}$ function.
\end{proof}

\begin{remark} \label{rqerr}
This bound of the error of approximation highlights two different sources of error: the error due to the perturbed data, and the error due to the lack of measurements for frequencies above $2k_{\text{max}}$. The uncertainty on the measurements can lead to small perturbations of the data, but the most important source of perturbation comes from the Born approximation and the error given in Proposition \ref{born}. \end{remark}

To recover the parameters of the bend from $f$, we see that 
\begin{equation*}
\int_0^1h_r(t)\dd t=1-\frac{1}{2(r+1)}-(r+1)\ln\left(\frac{r+1}{r}\right)=-\frac{1}{r}+o_{r\rightarrow +\infty}\left(\frac{1}{r}\right).
\end{equation*}
If $r$ is large enough, we can use the approximation $1/r$ or inverse the exact expression. The values of $x_c$ and $\theta$ are then deduced from the size of $\text{supp}(f)$. 
 
To conclude, with the measurements on a section of the waveguide of the scattered field due to a source $\tilde{s}_k=-2ik\delta_0(x)$ for every frequency in $(0,k_{\text{max}})$, we are able to reconstruct an approximation of $f$ from which we can derive the parameters of the bend. Moreover, we can quantify the error of this approximation, and this error decreases as $k_{\text{max}}$ increases and as $\theta$ decreases or $r$ increases. 

\begin{remark}\label{bendouble}
This inversion can also be applied for a succession of bends, each parametrized as in Figure \ref{bend}. In this case, the function $f$ is a sum of disjoint indicator functions. Our framework could also certainly be used to reconstruct more general geometries of bends. However, the expression of $S$ is then more complicated and the source function in \eqref{bendd} may no longer reduces to indicator function.
\end{remark}

\subsection{Detection of bumps}

We now consider shape defects as those depicted in Figure \ref{bumpp}: the goal is to reconstruct the functions $g$ and $h$ that define the bump geometries, from the measurements. 

\begin{figure}[h]

\begin{center}
\begin{tikzpicture}
\draw (0,1) -- (3,1);
\draw (5,1) -- (8,1);
\draw [dashed] (-1,1) -- (0,1);
\draw [dashed] (-1,0) -- (0,0);
\draw [dashed] (3,1)--(5,1);
\draw (0,0) -- (2,0);
\draw (7,0) -- (8,0);
\draw [dashed] (2,0) -- (7,0);
\draw [dashed] (8,0) -- (9,0);
\draw [dashed] (8,1) -- (9,1);
\draw [domain=3:5] [samples=200] plot (\x,{5/16*(\x-3)^2*(\x-5)^2+1});
\draw [domain=2:7] [samples=200] plot (\x,{1/4*sin(2*(\x-2)*3.14/5 r)});
\draw [->] [domain=0:2.5] [samples=200] plot (\x,{0.2*sin(10*(\x+0.6) r)+0.5});
\draw (-0.3,0.3) node[above]{$u_k^{\text{inc}}$};
\draw[<->] (4,1) -- (4, {5/16*(4-3)^2*(4-5)^2+1});
\draw[<->] (3,0) -- (3, {4/16*(3-2)^2*(3-4)^2});
\draw (4.2,0.7) node{$h(x)$};
\draw (3,-0.3) node{$g(x)$};

\end{tikzpicture} \caption{\label{bumpp} Representation of a shape defect in a waveguide.}\end{center}\end{figure}
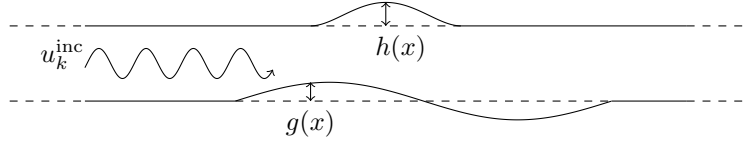

We assume that $\text{supp}(h),\text{supp}(g)$ are compact, that $1+h>g$, and that $h,g\in\mathcal{C}^2(\R)$ so Proposition~\ref{born} applies. Note that $h-1$ and $g$ do not need to be of constant sign. We define $\phi(x,y)=(x,(1+h(x)-g(x))y+g(x))$ and compute
\begin{equation}J\phi(x,y)=\left(\begin{array}{cc} 1 & 0 \\ (h'(x)-g'(x))y+g'(x) & 1+h(x)-g(x) \end{array}\right),\end{equation}
\begin{equation*} J\phi^{-1}(x,y)=\left(\begin{array}{cc} 1 & 0 \\ -\frac{(h'(x)-g'(x))y+g'(x)}{1+h(x)-g(x)} & \frac{1}{1+h(x)-g(x)} \end{array} \right).\end{equation*}
Moreover, $\tau=|\text{det}(J\phi)|=1+h(x)-g(x)$, $t_1=\sqrt{1+h'(x)^2}$, $t_2=\sqrt{1+g'(x)^2}$ and 
\begin{equation*}S=\left(\begin{array}{cc} 1+h(x)-g(x) & -(h'(x)-g'(x))y-g'(x) \\ -(h'(x)-g'(x))y-g'(x)  & \frac{\left((h'(x)-g'(x))y+g'(x)\right)^2}{1+h(x)-g(x)} +\frac{1}{1+h(x)-g(x)}\end{array}\right).\end{equation*}

Assuming that the bumps are located to the right of the section $\{0\}\times (0,1)$, we introduce a source $\tilde{s}_k=-2ik\delta_0(x)$, and notice that $s_k=\tilde{s}_k\circ\phi=\tilde{s}_k$. In the absence of defect, the wave field generated by this source would be $u_k^{\text{inc}}:=e^{ik|x|}$. Let $\tilde{u^s_k}:=\tilde{u_k}-u_k^{\text{inc}}$ be the scattered wave field which solves
\begin{equation}\left\{ \begin{array}{cl} \Delta \tilde{u_k} +k^2 \tilde{u_k} =\tilde{s}_k &  \text{ in } \widetilde{\Omega}, \\ \partial_\nu \tilde{u_k}=0 & \text{ on } \partial\widetilde{\Omega}, \\ \tilde{u_k} \text{ is outgoing.} \end{array}\right.  \end{equation}
Using the expression of $u_{k}^\text{inc}$ and the fact that if $x>0$ then $e^{ik|x|}=e^{ikx}$, $\tilde{u_k}^s$ satisfies the equation 
\begin{equation}\left\{\begin{array}{cl} \Delta \tilde{u_k}^s+k^2\tilde{u_k}^s=0 & \text{ in } \widetilde{\Omega}, \\ \partial_\nu \tilde{u_k}^s=\frac{h'(x)}{\sqrt{1+h'(x)^2}}ike^{ikx} & \text{ on } \partial\widetilde{\Omega}_{\text{top}}, \\ \partial_\nu \tilde{u_k}^s=\frac{-g'(x)}{\sqrt{1+g'(x)^2}}ike^{ikx}  & \text{ on } \partial\widetilde{\Omega}_{\text{bot}}, \\ \tilde{u_k}^s \text{ is outgoing.} \end{array} \right .\end{equation}
Transforming the deformed guide to a regular guide leads to
\begin{equation}\label{bump} \left\{\begin{array}{cl} \Delta u_k^s+k^2u_s=-\nabla\cdot(M\nabla u_k^s)-k^2\eps u_k^s &  \text{ in }\Omega, \\ \partial_\nu u_k^s= -M\nabla u_k^s \cdot \nu+h'(x)ike^{ikx} & \text{ on }\Omega_{\text{top}},\\ \partial_\nu u_k^s=-M\nabla u_k^s \cdot \nu -g'(x)ike^{ikx} & \text{ on }\Omega_{\text{bot}}, \\ u_k^s \text{ is outgoing}. \end{array} \right.\end{equation}
If the assumptions of Proposition \ref{born} are satisfied, $u_k^s$ is close to the solution $v_k$ of 
\begin{equation} \left\{\begin{array}{cl} \Delta v_k+k^2v_k=0 & \text{ in } \Omega, \\ \partial_\nu v_k=h'(x)ike^{ikx} & \text{ on } \Omega_{\text{top}},\\ \partial_\nu v_k= -g'(x)ike^{ikx}&  \text{ on } \Omega_{\text{bot}},  \\ v_k \text{ is outgoing.} \end{array} \right.\end{equation}
Given $k_{\text{max}}>0$, we measure the first mode $v_{k,0}$ of $v_k$ for all frequencies $k\in (0,k_{\text{max}})$. However, since we assumed that we can measure only propagative modes, we have access to $v_{k,1}$ the second mode of $v_k$ for all frequencies $k>\pi$, so for $k\in(\pi,k_{\text{max}})$. Using Proposition \ref{bord} and the inversion of source, we have access to 
\begin{equation}
v_{k,0}(0)=\frac{i}{2k}\int_0^{+\infty} (h'(z)-g'(z))ike^{ikz}e^{ikz}\dd z \qquad \forall k\in(0,k_{\text{max}}),
\end{equation}
\begin{equation}
v_{k,1}(0)=\frac{-i}{\sqrt{2}k_1}\int_0^{+\infty} (h'(z)+g'(z))ike^{ikz}e^{ik_1z}\dd z \qquad \forall k\in(\pi,k_{\text{max}}).
\end{equation}
We notice that 
\begin{equation} \label{gammabump1}
v_{k,0}(0)=2ik\Gamma(h'-g')(2k) \qquad \forall k\in(0,k_{\text{max}}),
\end{equation}
\begin{equation} \label{gammabump2}
v_{k,1}(0)=-\frac{\sqrt{2}ik(k_1+k)}{k_1}\Gamma(h'+g')(k+k_1) \qquad \forall k\in(\pi,k_{\text{max}}).
\end{equation}
We define $s_0=h'-g'$ and $s_1=h'+g'$. We have access to $\Gamma(s_0)(k)$ for all $k\in (0,2k_{\text{max}})$, and since $k\mapsto k+\sqrt{k^2-\pi^2}$ is one-to-one from $(\pi,k_{\text{max}})$ to $(\pi, k_{\text{max}}+\sqrt{k_{\text{max}}^2-\pi^2})$, we have access to $\Gamma(s_1)(k)$ for all $k\in (\pi,k_{\text{max}}+\sqrt{k_{\text{max}}^2-\pi^2})$. We denote by $d_0(k)=\Gamma(s_0)(k)$, $d_1(k)=\Gamma(s_1)(k)$ the data and consider the perturbed data $d_{0_{\text{app}}}$, $d_{1_{\text{app}}}$. The method described in Section 2 provides approximations $s_{0_{\text{app}}}$, $s_{1_{\text{app}}}$ which we can control by the following:

\begin{proposition} \label{contr2}
Let $s_0,s_1,s_{0_{\text{app}}}, s_{1_{\text{app}}} \in \text{H}^1(-r,r)$ where $r\in \R^*_+$. Let $k_{\text{max}}\in R^*_+$, $d_0=\Gamma(s_0)$, $d_{0_{\text{app}}}=\Gamma(s_{0_{\text{app}}})$ defined on $(0,2k_{\text{max}})$, $d_1=\Gamma(s_1)$, $d_{1_{\text{app}}}=\Gamma(s_{1_{\text{app}}})$ defined on $(\pi, k_{\text{max}}+\sqrt{k_{\text{max}}^2-\pi^2})$. Assume that there exists $M\in \R^*_+$ such that $\Vert s_i \Vert_{\text{H}^1(-r,r)}\leq M$ and $\Vert s_{i_{\text{app}}} \Vert_{\text{H}^1(-r,r)}\leq M$ for $i=0,1$. Then for every $0<\eps<1$, there exists a constant $\xi_{k_{\text{max}}}$, depending on $r,M,\eps$, such that 
\begin{equation} \Vert s_0-s_{0_{\text{app}}} \Vert_{\text{L}^2(-r,r)}^2\leq \frac{4}{\pi}\Vert d_0-d_{0_{\text{app}}}\Vert_\text{H}^2+\frac{2\pi}{k_{\text{max}}^2}M^2, \end{equation}
\begin{equation}\begin{split}
\Vert s_1-s_{1_{\text{app}}} \Vert_{\text{L}^2(-r,r)}^2\leq \xi_{k_{\text{max}}}\Vert d_1-d_{1_{\text{app}}}\Vert^{2-2\eps}_{\text{H}}&+\frac{4}{\pi}\Vert d_1-d_{1_{\text{app}}}\Vert_\text{H}^2 \\ &\quad +\frac{8\pi}{\left(k_{\text{max}}+\sqrt{k_{\text{max}}^2-\pi^2}\right)^2}M^2.\end{split}
\end{equation}
\end{proposition}
\begin{proof} Noticing that $|\F(h')(k)|=2|k\Gamma(h')(k)|$, we apply Theorem \ref{therr} and Remark \ref{remcite} with $\omega_0=0$ and $\omega_1=2k_{\text{max}}$, and then $\omega_0=\pi$ and $\omega_1=k_{\text{max}}+\sqrt{k_{\text{max}}^2-\pi^2}$.
\end{proof}
\begin{remark} \label{rqerr2}
This estimate highlights the different sources of error: the error due to the perturbed data, the error due to the lack of measurements at high frequencies, and due to the lack of measurements for the low frequencies of $s_1$. Note that the error diminishes if $K$ increases and if the bump gets smaller. Numerical illustrations can be found in Section 4.4. \end{remark}

\subsection{Detection of inhomogeneities}

This case is different from the two previous cases, as the presence of an  inhomogeneity affects the index of the medium and leads to changes in the homogeneous Helmholtz equation:
\begin{equation} \Delta u+k^2(1+h(x,y))u=0. \end{equation}
We assume that $\text{supp}(h)$ is compact and that the inhomogeneity is located to the right of the section $\{0\}\times(0,1)$. To detect the defect, we introduce a source $s_k=-2ik\delta_0(x)$. In the absence of defect, the wave field generated by this source would be $u_k^{\text{inc}}:=e^{ik|x|}$. Let $u^s_k$ be the scattered wave field defined by $u^s_k:=u_k-u_k^{\text{inc}}$. We know that $u_k$ satisfies the equation \eqref{directsource}, and so
\begin{equation} \label{eqinhomo} \left\{\begin{array}{cl} \Delta u^s_k+k^2u^s_k=-k^2hu^\text{inc}_k-k^2hu^{\text{inc}}_k &  \text{ in } \Omega, \\
\partial_\nu u^s_k=0 &  \text{ on } \partial\Omega, \\
u^s_k \text{ is outgoing}. &
\end{array} \right.\end{equation} 
Let $\mathcal{S}(u):=k^2hu$ which satisfies the hypothesis of Proposition \ref{borninhomo}, and for every $r>0$, 
\begin{equation}\Vert \mathcal{S} \Vert_{\text{H}^2(\Omega_r)\rightarrow\text{L}^2(\Omega_r)} \leq k^2\Vert h \Vert_{\text{L}^\infty(-r,r)}.\end{equation}
 Proposition \ref{borninhomo} shows that if $k^2\Vert h \Vert_{\text{L}^\infty(-r,r)}$ is small enough, $u^s_k$ is close to $v_k$ the solution of 
\begin{equation} \left\{\begin{array}{cl} \Delta v_k+k^2v_k=-k^2hu_k^{\text{inc}} &  \text{ in } \Omega, \\
\partial_\nu v_k=0 &  \text{ on } \partial\Omega, \\
v_k \text{ is outgoing}. &
\end{array} \right.\end{equation} 
and that, with $C$ the constant defined in Proposition \ref{directg},
\begin{equation} 
\Vert u-v\Vert_{\text{H}^2(-r,r)}\leq \frac{C^2 k^4\Vert h\Vert^2_{\text{L}^\infty(-r,r)}\Vert u_k^{\text{inc}}\Vert_{\text{L}^2(-r,r)}}{1-Ck^2\Vert h\Vert_{\text{L}^\infty(-r,r)}}.\end{equation}

We assume that the measurements consist in the $n$-th propagative mode $v_{k,n}$ for all frequencies $k\in (0,k_{\text{max}})$ where $k_{\text{max}}>0$ is given. Proposition \ref{source} shows that for every $k\in (0,k_{\text{max}})$,
\begin{equation} \label{gammainhomo}
v_{n,k}(0)=\frac{i}{2k_n} \int_{0}^{+\infty} k^2h_n(z)e^{ikz}e^{ik_nz}\dd z=\frac{(k+k_n)k^2}{k_n}\Gamma(h_n)(k+k_n).
\end{equation}
Since we assume that only the propagative modes are measured, the frequency $k$ must satisfy $k>n\pi$, and $k_n\in \R$. The function $k\mapsto k+\sqrt{k^2-n^2\pi^2}$ is one-to-one from $(n\pi,k_{\text{max}})$ to $(n\pi,k_{\text{max}}+\sqrt{k_{\text{max}}^2-n^2\pi^2})$. This means that we have access to $\Gamma(h_n)(k)$ for every $k\in (n\pi,k_{\text{max}}+\sqrt{k_{\text{max}}^2-n^2\pi^2})$. We denote by $d_n=\Gamma(h_n)$ the data and by $d_{n_{\text{app}}}$ the perturbed data. We use the method described in section 2 to reconstruct $h_{n_{\text{app}}}$, an approximation of $h_n$, and control the error using Theorem~\ref{therr}. 

\begin{proposition} \label{contr3}
Let $n\in \N$, $h_{n_{\text{app}}}, h_n\in \text{H}^1(-r,r)$ where $r>0$. Let $k_{\text{max}}>n\pi$, $d_n(k)=\Gamma(h_n)(k)$ and $d_{n_{\text{app}}}(k)=\Gamma(h_{n_{\text{app}}})(k)$ for $k\in(n\pi,k_{\text{max}}+\sqrt{k_{\text{max}}^2-n^2\pi^2})$. We assume that there exists $M\in \R^*_+$ such that $\Vert h_n\Vert_{\text{H}^1(-r,r)}\leq M$, $\Vert h_{n_{\text{app}}} \Vert_{\text{H}^1(-r,r)}\leq M$. Then for every $0<\eps<1$ there exists a constant $\xi_{n,k_{\text{max}}}$ depending on $r,M,\eps$ such that 
\begin{equation}\begin{split} 
\Vert h_n-h_{n_{\text{app}}}\Vert_{\text{L}^2(-r,r)}^2\leq \xi_{n,k_{\text{max}}}\Vert d_n-d_{n_{\text{app}}}\Vert^{2-2\eps}_{\text{H}}+\frac{4}{\pi}\Vert d_n-d_{n_{\text{app}}}\Vert_\text{H}^2\hspace{2cm}\\ +\frac{8\pi}{\left(k_{\text{max}}+\sqrt{k_{\text{max}}^2-n^2\pi^2}\right)^2}M^2.\end{split} 
\end{equation}
\end{proposition}

\begin{corollary} \label{corocontr}
Let $k_{\text{max}}\in \R^*_+$ and $N\in \N$ such that $N<k_{\text{max}}/\pi$. Let $h_{\text{app}},h\in \text{H}^1(\Omega_r)$ where $r>0$ and $d=F(h)$, $d_{\text{app}}=F_s(h_{\text{app}})$ such that $d_n$ and $d_{n_{\text{app}}}$ are defined on $(n\pi,k_{\text{max}}+\sqrt{k_{\text{max}}^2-n^2\pi^2})$. We assume that there exists $M\in \R^*_+$ such that $\Vert h\Vert_{\text{H}^1(\Omega_r)}\leq M$ and $\Vert h_{\text{app}} \Vert_{\text{H}^1(\Omega_r)}\leq M$. Then for every $0<\eps<1$ there exists a constant $\xi_{N,k_{\text{max}}}$ depending on $r,M,\eps$ such that
\begin{equation}\begin{split}
\Vert h-h_{\text{app}}\Vert_{\text{L}^2(\Omega_r)}^2\leq \xi_{N,k_{\text{max}}}(N+1)^\eps \Vert d-d_{\text{app}}\Vert_{\ell^2(H)}^{2-2\eps}+\frac{4}{\pi}\Vert d-d_{\text{app}}\Vert^2_{\text{l}^2(H)}\\
+\frac{8\pi (N+1)}{K^2}M^2+\frac{4}{N^2\pi^2}M^2.\end{split}
\end{equation}
\end{corollary}

\begin{proof} Using the previous proposition,
\begin{equation*}\begin{split}
\Vert h-h_{\text{app}}\Vert_{\text{L}^2(\Omega_r)}^2\leq \sum_{n=0}^N \xi_{n,k_{\text{max}}}\Vert d_n-d_{n_{\text{app}}}\Vert^{2-2\eps}_{\text{H}}+\frac{4}{\pi}\Vert d_n-d_{n_{\text{app}}}\Vert_\text{H}^2\qquad \qquad \\\qquad +\frac{8\pi M^2}{\left(k_{\text{max}}+\sqrt{k_{\text{max}}^2-n^2\pi^2}\right)^2}+\sum_{n>N}\Vert h_n-h_{n_{\text{app}}}\Vert^2_{\text{L}^2(-r,r)}.\end{split}
\end{equation*}
We define $\xi_{N,k_{\text{max}}}=\max_{n=0,..,N}\xi_{n,k_{\text{max}}}$ and using the concavity of $x\mapsto x^{1-\eps}$, we deduce that 
\begin{equation*}\begin{split}
\Vert h-h_{\text{app}}\Vert_{\text{L}^2(\Omega_r)}^2\leq \xi_{N,k_{\text{max}}}(N+1)^\eps \Vert d-d_{\text{app}}\Vert_{\ell^2(H)}^{2-2\eps}+\frac{4}{\pi}\Vert d-d_{\text{app}}\Vert^2_{\ell^2(H)}\\+\frac{8\pi M^2 (N+1)}{k_{\text{max}}^2}+\frac{\Vert\partial_y(h-h_{\text{app}})\Vert^2_{\text{L}^2(\Omega_r)}}{N^2\pi^2}.\end{split}
\end{equation*}
We conclude using the upper bound on $\Vert h\Vert_{\text{H}^1(\Omega_r)}$ and $\Vert h_{\text{app}}\Vert_{\text{H}^1(\Omega_r)}$. 
\end{proof}

\begin{remark}Again, this estimate highlights the different sources of error: the lack of measurements if the mode if greater than $1$ in the low frequencies, the perturbed data, the lack of measurements in the high frequencies and finally the truncation to the $N$-th mode.  The predominant term here seems to be the first one, and we need to find a balance between increasing $N$ to decrease the error of truncation and diminishing $N$ to lower the value of $\xi_{N,k_{\text{max}}}$. 
\end{remark}

Unlike the two previous cases, the detection of inhomogeneities requires more modes than just the first two modes. However, using measurements on one section of the scattered field associated with a source $s_k=-2ik\delta_0(x)$ allows reconstruction of an approximation of $h$ with quantified error.

\section{Numerical Results}

\subsection{Numerical source inversion from limited frequency data}

In Proposition \ref{invgamma}, we have seen that the forward modal operator $\Gamma$ is inversible. Knowing the measurements of the wavefield generated by a source for every frequency, we are theoretically abble to reconstruct the source. Moreover, Theorem \ref{therr} shows that if the source is compactly supported, measurements are only needed for a finite interval of frequencies to approximate the source. In this section, we discuss the numerical aspects of the inversion. 

We assume that the wavefield in the waveguide is generated by a source $f$ compactly supported, located between the sections $x=x_m$ and $x=x_M$. The interval $[x_m,x_M]$ is regularly discretized by a set of $N_X$ values $X$, and seek an approximation of $f(X)$. The measurements of the wavefield are made for a discrete set of $N_K$ frequencies denoted $K$. Let $h=\frac{x_M}{N_X-1}$ denote the stepsize of the discretization $X$. Using Definition \ref{defigamma} and Equation \eqref{defgamma}, the operator $f\mapsto (k\mapsto k\Gamma(f)(k))$ maps $\text{L}^2(\R)$ onto $\text{L}^2(\R^*_+)$ and can be discretized by the operator
\begin{equation} \gamma: \begin{array}{rcl} \C^{N_X} & \rightarrow & \C^{N_K} \\ y & \mapsto & \left(\frac{ih}{2}\displaystyle\sum_{x\in X}y_xe^{ikx}\right)_{k\in K} \end{array}.\end{equation}

To invert this operator, we use a least square method. Given the data $d=\gamma(f(X))$, we seek an approximation of $f(X)$ by minimizing the quantity 
\begin{equation*}
\frac{1}{2} \Vert \gamma(y)-d\Vert_{\ell^2\left(\C^{N_X}\right)}^2.
\end{equation*}
To avoid small oscillations in the reconstruction we also define the discrete gradient

\begin{equation} G : \begin{array}{rcl}
\C^{N_X} & \rightarrow & \C^{N_X} \\ y & \mapsto & \left(y_i-y_{i-1}\right)_{1\leq i \leq N_X} 
\end{array}, \end{equation}
with the convention that $x_0=x_{N_X}$ and $x_{N_X+1}=x_1$. Note that the adjoints of $\gamma$ and $G$, denoted by $\gamma^*$ and $G^*$, can be easily computed. For $\lambda>0$, we minimize the quantity
\begin{equation} J(y)=\frac{1}{2}\Vert \gamma(y)-d\Vert_{\ell^2\left(\C^{N_K}\right)}^2+\frac{\lambda}{2}\Vert G(y)\Vert_{\ell^2\left(\C^{N_X}\right)}^,.\end{equation}
with a steepest descent method, with the initialization $y_0=~(0)_{x\in X}$:
\begin{equation} \label{algo} y_{m+1}=y_m-\frac{\Vert \nabla J(y_m)\Vert^2_{\ell^2\left(\C^{N_X}\right)}}{\Vert S(\nabla J(y_m))\Vert_{\ell^2\left(\C^{N_K}\right)}^2+\Vert G(\nabla J(y_m))\Vert_{\ell^2\left(\C^{N_X}\right)}^2}\nabla J(y_m),\end{equation}
where 
\begin{equation}
\nabla J(y_m)=\gamma^*(\gamma(y_m)-d)+\lambda G^*(G(y_m)).
\end{equation}

We use this algorithm to illustrate the results given in Theorem \ref{therr}. Firstly, we reconstruct a source $f$ with a gap in the high frequencies, \textit{i.e.} for which measurements of the wavefield generated by $f$ are available for a discrete set of frequencies between $0$ and $\omega_1$. Figure \ref{cuthf} presents the comparison between a function $f$ and its reconstruction for different values of $\omega_1$. As expected, we observe convergence when $\omega_1$ increases, and the reconstruction becomes almost perfect visually. The speed of convergence is illustrated in Figure \ref{cuthferr}, and as expected from Theorem \ref{therr}, the $\text{L}^2$-error between the function and its approximation decreases like $1/\omega_1$.

\begin{figure}[h!]
\centering
\resizebox{4.8in}{!}{\input{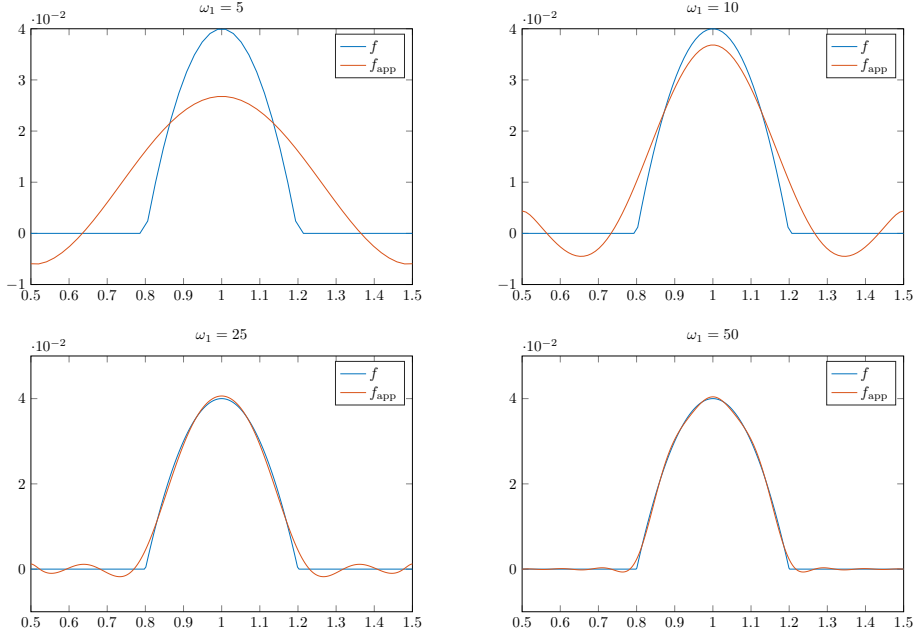}}
\caption{\label{cuthf} Reconstruction of $f(x)=(x-0.8)(1.2-x)\textbf{1}_{0.8\leq x\leq 1.2}$ for different values of $\omega_1$ using the discrete operator $\gamma$ and the algorithm \eqref{algo} with $\lambda=0.001$. Here, $X$ is the discretization of $[0.5,1.5]$ with $10\omega_1$ points, and $K$ is the discretization of $[0.01,\omega_1]$ with $1000$ points.}
\end{figure}

\begin{figure}[h!]
\centering
\scalebox{.55}{
%
%
\definecolor{mycolor1}{rgb}{0.00000,0.44700,0.74100}%
\begin{tikzpicture}

\begin{loglogaxis}[%
width=5.694in,
height=2.683in,
at={(0in,0in)},
scale only axis,
grid=minor,
xmin=3,
xmax=60,
ymin=0.0001,
ymax=0.02,
xlabel style={font=\color{white!15!black}},
xlabel={$\omega_1$},
ylabel style={font=\color{white!15!black}},
ylabel={$\Vert f-f_{\text{app}}\Vert_2$},
axis background/.style={fill=white},
legend style={legend cell align=left, align=left, draw=white!15!black}
]
\addplot [color=mycolor1]
  table[row sep=crcr]{%
3	0.0102507715285982\\
4.96551724137931	0.0083689547605327\\
6.93103448275862	0.00534282216102071\\
8.89655172413793	0.00460969899221156\\
10.8620689655172	0.00288749666555148\\
12.8275862068966	0.00189465294664698\\
14.7931034482759	0.00169539939495101\\
16.7586206896552	0.00159003400080123\\
18.7241379310345	0.00159294150255294\\
20.6896551724138	0.00155332704253738\\
22.6551724137931	0.00138800917143074\\
24.6206896551724	0.00120016419649199\\
26.5862068965517	0.00105722571433215\\
28.551724137931	0.00086834676866764\\
30.5172413793103	0.00075056184522202\\
32.4827586206897	0.000735065199598808\\
34.448275862069	0.000737249600843856\\
36.4137931034483	0.000730912238858634\\
38.3793103448276	0.000701236580641188\\
40.3448275862069	0.00064482132322914\\
42.3103448275862	0.000560133144945395\\
44.2758620689655	0.000501087764458307\\
46.2413793103448	0.000467149698138688\\
48.2068965517241	0.000446776592657412\\
50.1724137931034	0.000447703775102475\\
52.1379310344828	0.000446834217839432\\
54.1034482758621	0.000433499956414707\\
56.0689655172414	0.000403011818516453\\
58.0344827586207	0.000375943048470542\\
60	0.000342039162098546\\
};
\addlegendentry{$\Vert f-f_{\text{app}}\Vert_2$}

\addplot [color=black, dashed]
  table[row sep=crcr]{%
3	0.0165956894559546\\
4.96551724137931	0.0100265623796393\\
6.93103448275862	0.00718320886899529\\
8.89655172413793	0.00559622086305447\\
10.8620689655172	0.00458357137354938\\
12.8275862068966	0.00388124995340875\\
14.7931034482759	0.00336555940015863\\
16.7586206896552	0.00297083329767089\\
18.7241379310345	0.00265897786863362\\
20.6896551724138	0.00240637497111342\\
22.6551724137931	0.00219760271334559\\
24.6206896551724	0.00202216384127178\\
26.5862068965517	0.00187266534716998\\
28.551724137931	0.0017437499790677\\
30.5172413793103	0.00163144065838198\\
32.4827586206897	0.00153272291153721\\
34.448275862069	0.00144527025292098\\
36.4137931034483	0.00136725850631445\\
38.3793103448276	0.00129723718119322\\
40.3448275862069	0.00123403844672483\\
42.3103448275862	0.00117671147731708\\
44.2758620689655	0.00112447428556702\\
46.2413793103448	0.00107667783942435\\
48.2068965517241	0.00103277895755941\\
50.1724137931034	0.000992319575716876\\
52.1379310344828	0.000954910702822787\\
54.1034482758621	0.000920219874230755\\
56.0689655172414	0.000887961243953293\\
58.0344827586207	0.000857887690236515\\
60	0.000829784472797732\\
};
\addlegendentry{$-1$ slope}

\end{loglogaxis}
\end{tikzpicture}%
}
\caption{\label{cuthferr} $\text{L}^2$-error between $f(x)=(x-0.8)(1.2-x)\textbf{1}_{0.8\leq x\leq 1.2}$ and its reconstruction $f_{app}$ for different values of $\omega_1$ using the discrete operator $\gamma$ and the algorithm \eqref{algo} with $\lambda=0.001$. Here, $X$ is the discretization of $[0.5,1.5]$ with $10\omega_1$ points, and $K$ is the discretization of $[0.01,\omega_1]$ with $1000$ points.}
\end{figure}

Secondly, we investigate the influence of $r$ and $\omega_0$ in Theorem \ref{therr}. Consistently with Propositions \ref{contr2} and \ref{contr3}, we choose $\omega_0$ to be a multiple of $\pi$. In Figure \ref{cutbfcut}, we present the comparison between a 1D function $f$ and its reconstruction is represented for different values of $\omega_0$, when the support of $f$ is fixed. Figure \ref{cutbfr} depicts the comparison between a 1D function $f$ and its reconstruction for different sizes $r$ of support when $\omega_0$ is fixed. As expected from the definition of the constant $C$ in Theorem \ref{therr}, the quality of the reconstruction deteriorates when $r$ and $\omega_0$ increase.

\begin{figure}[h!]
\centering
\resizebox{4.8in}{!}{\input{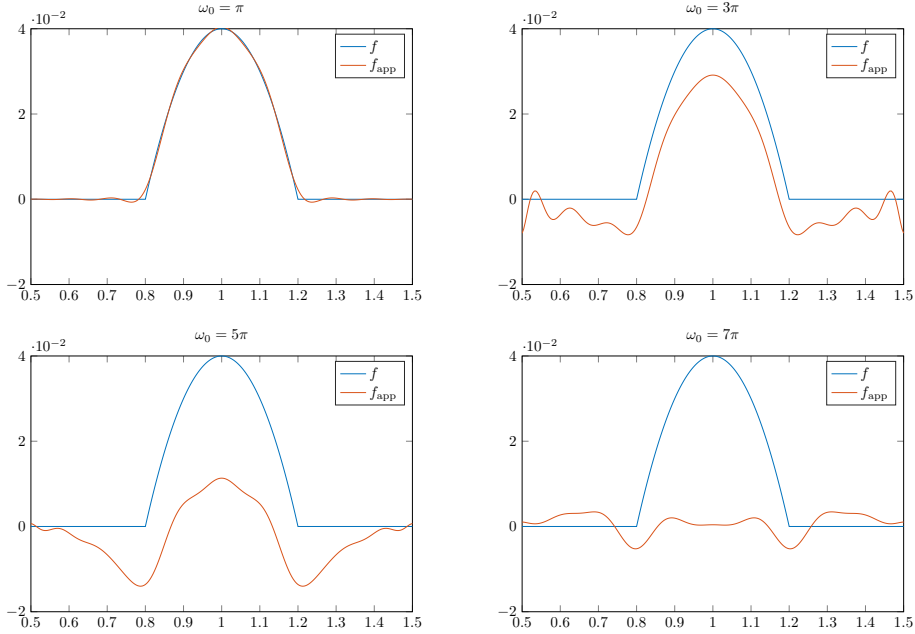}}
\caption{\label{cutbfcut} Reconstruction of $f(x)=(x-0.8)(1.2-x)\textbf{1}_{0.8\leq x\leq 1.2}$ for different values of $\omega_0$ and $r=0.5$ using the discrete operator $\gamma$ and the algorithm \eqref{algo} with $\lambda=0.001$. Here, $X$ is the discretization of $[0.5,1.5]$ with $251$ points, and $K$ is the discretization of $[\omega_0,50]$ with $1000$ points.}
\end{figure}

\begin{figure}[h!]
\centering
\resizebox{4.8in}{!}{\input{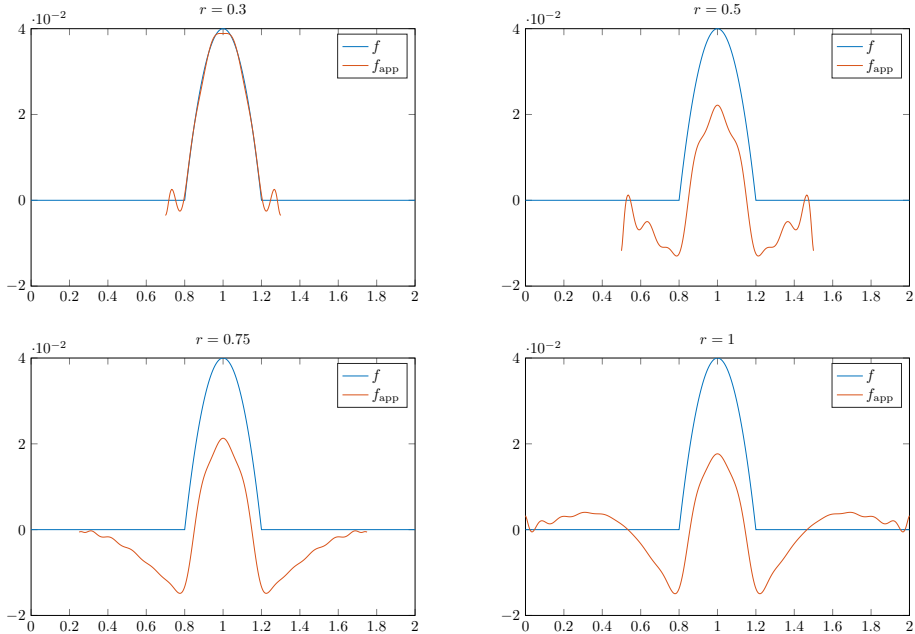}}
\caption{\label{cutbfr} Reconstruction of $f(x)=(x-0.8)(1.2-x)\textbf{1}_{0.8\leq x\leq 1.2}$ for different sizes of support $r$ and $\omega_0=3\pi$ using the discrete operator $\gamma$ and the algorithm \eqref{algo} with $\lambda=0.001$. Here, $X$ is the discretization of $[1-r,1+r]$ with $500r+1$ points, and $K$ is the discretization of $[3\pi,50]$ with $1000$ points.}
\end{figure}

The reconstruction of a source is almost perfect for $\omega_0=0$ if we increase sufficiently $\omega_1$. However, if $\omega_0>0$, the problem is ill-conditioned and if the size of the support of the source or $\omega_0$ increase, the quality of the reconstruction is poor.

\subsection{Generation of data for the detection of defects}

Applying the results of Section 3 requires measurements generated by a defect on a section of the wavefield. This data is generated by solving numerically the PDE with Matlab, and evaluating its solution on a section of the waveguide. The equations of propagation in a regular waveguide $\Omega$ for a bend, a bump and a inhomogeneity are given by \eqref{bendd}, \eqref{bump} and \eqref{eqinhomo} respectively. In the following, we  assume that the interesting part of the waveguide is located between $x=0$ and $x=8$, and that the measurements are made on the section $\{1\}\times (0,1)$. To generate the solution of these equations of propagation on $[0,8]\times [0,1]$, we use the finite element method and a perfectly matched layer between $x=-19$ and $x=0$  on the left side of the waveguide and between $x=8$ and $x=27$ on the right side. The coefficient of absorption for the perfectly matched layer is defined by $-k((x-8)\textbf{1}_{x\geq 8}-x\textbf{1}_{x\leq 0})$. The structured mesh is built with a stepsize \textcolor{red}{$0.01$.} 

\subsection{Detection of bends}

Using the method described in the previous subsection, we generate the solution of \eqref{bendd} for a set of frequencies $K$ and we evaluate the solutions on the section $\{1\}\times (0,1)$. As explained in section 3.2 and equation \eqref{gammabend}, the corresponding data amounts to knowing $\Gamma(s)(2k)$ for every $k\in K$, where 
\begin{equation} 
s=\textbf{1}_{x\in [x_c,x_c+\theta(r+1)]}\left(1-\frac{1}{2(r+1)}-(r+1)\ln\left(\frac{r+1}{r}\right)\right). \end{equation}
Note that algorithm \eqref{algo} could be used to construct an approximation $s_{\text{app}}$ of $s$. However, since we are looking for a rectangular function, we can directly define $s_{\text{app}}=-p_1\textbf{1}_{x\in[p_2,p_2+p_3]}$ and see that 
\begin{equation}\label{bendapp}\Gamma(s_{\text{app}})(k)= -\frac{ip_1}{k}e^{ik\frac{2p_2+p_3}{2}}\sin\left(\frac{p_3}{2}k\right).\end{equation}
We determine $(p_1,p_2,p_3)$ by minimizing $\Vert \Gamma(s_{\text{app}})(2k)-\Gamma(s)(2k)\Vert_{\ell^2\left(\C^{N_K}\right)}$, and the approximations of $x_c$, $r$ and $\theta$ follow. We present in Figure \ref{bendrec} the reconstructions of two different bends, and in Table \ref{tableau} the relative error on the estimation of $(x_c,r,\theta)$ for different bends. We note that if the bend is really small, the reconstruction is very good. On the other hand, when $r$ increases  or when $\theta$ decreases, the reconstruction deteriorates due to the fact that the Born approximation is no longer a good approximation of the wavefield in the waveguide. As mentionned in Remark \ref{bendouble}, our algorithm can also be used to recover a succession of bends, as shown in Figure \ref{bendrec2}.

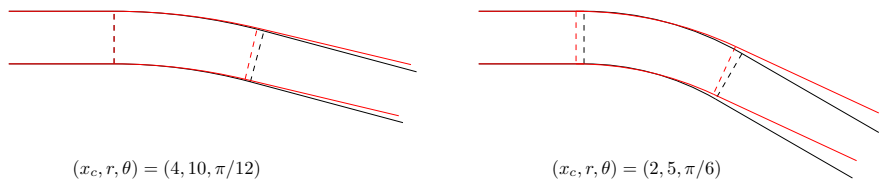
\begin{figure}[h]
\begin{center}
\scalebox{.7}{\begin{tikzpicture}
\draw (1,1) -- (3,1);
\draw (1,0) -- (3,0);
\draw [dashed] (3,0) -- (3,1);
\draw [dashed] (5.5882,-0.3407)--(5.8470,0.6252);
\draw (3,1) arc(90:75:11);
\draw (5.5882,-0.3407)-- ++ (2.8978,-0.7765);
\draw (5.8470,0.6252)-- ++ (2.8978,-0.7765);
\draw (3,0) arc(90:75:10);
\draw [color=red](1,1) -- (3,1);
\draw [color=red](1,0) -- (3,0);
\draw [color=red] [dashed] (3,0) -- (3,1);
\draw [color=red](3,1) arc(90: 76.61033:11.7460);
\draw [color=red] [dashed](5.4885,-0.2921) --(5.7200,0.6807); 
\draw [color=red](5.4885,-0.2921)-- ++ (2.9185,-0.6947);
\draw [color=red](5.7200,0.6807)-- ++ (2.9185,-0.6947);
\draw [color=red](3,0) arc(90: 76.6103:10.7460);
\draw (4,-2) node{$(x_c,r,\theta)=(4,10,\pi/12)$};
\end{tikzpicture}} \hspace{5mm}  \scalebox{0.7}{
\begin{tikzpicture}

\draw (0,1) -- (2,1);
\draw (0,0) -- (2,0);
\draw [dashed] (2,0)--(2,1); 
\draw (2,1) arc(90:60:6);
\draw [dashed] (5,0.1962)--(4.5,-0.6699); 
\draw (5,0.1962)-- ++ (2.5981,-1.5);
\draw (4.5,-0.6699)-- ++ (2.5981,-1.5);
\draw (2,0) arc(90:60:5);
\draw [color=red](0,1) -- (2,1);
\draw [color=red](0,0) -- (2,0);
\draw [color=red] [dashed] (1.8485,1)--(1.8485,0) ; 
\draw [color=red](1.8485,1) arc(90: 65.0778: 7.1882);
\draw [color=red] [dashed]  (4.8775,0.3306)--(4.4561,-0.5762); 
\draw [color=red](4.8775,0.3306)-- ++ (2.7206,-1.2642);
\draw [color=red](4.4561,-0.5762)-- ++ (2.7206,-1.2642);
\draw [color=red](1.8485,0) arc(90:65.0778: 6.1882);
\draw (3,-2) node{$(x_c,r,\theta)=(2,5,\pi/6)$};

\end{tikzpicture}}
\caption{\label{bendrec} Reconstruction of two different bends. The black lines represent the initial shape of $\Omega$, and the red the reconstruction of $\Omega$. In both cases, $K$ is the discretization of $[0.01,40]$ with $100$ points, and the reconstruction is obtain by \eqref{bendapp}. On the left, the initial parameters of the bend are $(x_c,r,\theta)=(4,10,\pi/12)$ and on the right, $(x_c,r,\theta)=(2,5,\pi/6)$. } \end{center} \end{figure}

\begin{table}[h]
\begin{center}
\begin{tabular}{|c|c|c|c|} \hline  $(x_c,r,\theta)$ & $(2.5,40,\pi/80)$ & $(4,10,\pi/12)$ &  $(2,5,\pi/6)$ \\ \hline relative error on $x_c$ & $1.8\%$ & $0\%$  &$7.6\%$ \\ \hline relative error on $r$ & $3.0\%$ & $7.5\%$ & $23.8\%$ \\ \hline relative error on $\theta$ & $1.6\%$ & $10.7\%$ & $16.9\%$\\  \hline \end{tabular} \vspace{3mm}
\end{center}
\caption{\label{tableau} Relative errors on the reconstruction of $(x_c,r,\theta)$ for different bends. In each case, $K$ is the discretization of $[0.01,40]$ with $100$ points, and the reconstruction is obtain by \eqref{bendapp}.}
\end{table}

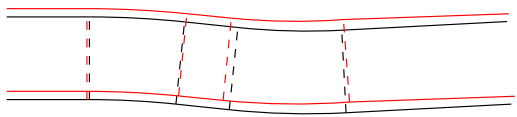
\begin{figure}[h]
\begin{center}
\scalebox{1.1}{\begin{tikzpicture}
\draw (1,1) -- (2,1);
\draw (1,0) -- (2,0);

\draw (2,1) arc(90:84:11);
\draw [dashed] (2,0) --(2,1); 
\draw [dashed] (3.0453,-0.0548)-- (3.1498,0.9397); 
\draw (3.0453,-0.0548)-- ++ (0.6445,-0.0677);
\draw (3.1498,0.9397)-- ++ (0.6445,-0.0677);
\draw (2,0) arc(90:84:10);
\draw [dashed]  (3.6898,-0.1225) -- (3.7943,0.8720); 
\draw (3.6898,-0.1225) arc (-96:-87:9) node (a){} -- ++(1.9973,0.1047); 
\draw (3.7943,0.8720) arc (-96:-87:8) node (b) {} -- ++(1.9973,0.1047); 
\draw [dashed] (a.center)--(b.center);

\draw [color=red](1,1.1) -- (1.97,1.1);
\draw [color=red](1,0.1) -- (1.97,0.1);
\draw [color=red] [dashed] (1.97,0) --(1.97,1); 
\draw [color=red](1.97,0.1) arc (90:84.2704:11.1) node (c){} -- ++ (0.5373,-0.0539); 
\draw [color=red] (1.97,1.1) arc (90:84.2704:12.1) node(d) {}-- ++ (0.5373,-0.0539); 
\draw [color=red] [dashed] (3.6155, -0.0094) -- (3.7153,0.9857); 
\draw [color=red] (3.6155, -0.0094) arc(-95.7296:-85.7888:8.82) node(e) {}--++ (1.9988,0.0700); 
\draw [color=red] (3.7153,0.9857) arc (-95.7296:-85.7888:7.82) node(f) {} --++ (1.9988,0.0700); 
\draw [color=red] [dashed] (c.center)--(d.center); 
\draw [color=red] [dashed] (e.center)--(f.center); 
\end{tikzpicture}} \caption{\label{bendrec2} Reconstruction of a waveguide with two successive bends. The black lines represent the initial shape of $\Omega$, and the red the reconstruction of $\Omega$, slightly shifted for comparison purposes. In both cases, $K$ is the discretization of $[0.01,40]$ with $100$ points. The parameters of the two bends are $(x_c^{(1)},r^{(1)},\theta^{(1)})=(2,10,\pi/30))$ and $(x_c^{(2)},r^{(2)},\theta^{(2)})=(3.8,8,-\pi/20))$} \end{center} \end{figure}

\subsection{Detection of bumps}

Using the method described in section 4.2, we generate the solutions of \eqref{bump} for a set of frequencies $K$ and we evaluate the solutions on the section $\{1\}\times (0,1)$. In view of Remark \ref{distpi}, and to ensure that the Born hypothesis \eqref{hypborn} is satisfied, we do not choose frequencies in $[n\pi-0.2,n\pi+0.2]$, for every $n\in \N$.  As explained in Section 3.3 and equations \eqref{gammabump1}, \eqref{gammabump2}, the data only determines $\Gamma(s_0)(2k)$ for every $k\in K$ and $\Gamma(s_1)(k+\sqrt{k^2-\pi^2})$ for every $k\in K$, $k>\pi$, where $h$ and $g$ paramatrize the bump (recall that $s_0=h'+g'$, $s_1=-\sqrt{2}h'+g'$). Using the algorithm \eqref{algo}, we find an approximation of $h'$ and $g'$, and the approximation of $h$ and $g$ follows by integration. In figure \ref{testbump1}, we represent two different reconstructions of a shape defect. As predicted  in Proposition \ref{born}, the reconstruction improves when $\Vert h\Vert_{\mathcal{C}^1(\R)}$ and $\Vert g \Vert_{\mathcal{C}^1(\R)}$ decrease. Table \ref{tableau2} illustrates this point as it depicts the relative error on a reconstruction of $h$ when its amplitude increases.  

\begin{figure}[h]
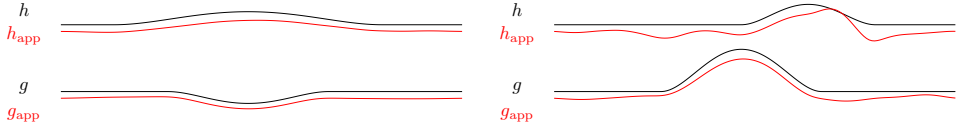

\centering
\scalebox{0.7}{\input{testbump1}}\hspace{3mm} \scalebox{0.7}{\input{testbump2}} 
\caption{\label{testbump1} Reconstruction of two shape defects. In black, the initial shape of $\Omega$, and in red the reconstruction, slightly shifted for comparison purposes. In both cases, $K$ is the discretization of $[0.01,70]\setminus \{[n\pi-0.2,n\pi+0.2], n\in \N\}$ with $300$ points, $X$ is the discretization of $[3,4.5]$ with $151$ points and we use the algorithm \eqref{algo} with $\lambda=0.08$ to reconstruct $s_0$ and $s_1$. On the left, $h(x)=\frac{5}{16}\textbf{1}_{3.2\leq x\leq 4.2}(x-3.2)^2(4.2-x)^2$ and $g(x)=-\frac{35}{16}\textbf{1}_{3.4\leq x\leq 4}(x-3.4)^2(4-x)^2$. On the right, $h(x)=\frac{125}{16}\textbf{1}_{3.7\leq x\leq 4.2}(x-3.7)^2(4.2-x)^2$ and $g(x)=\frac{125}{16}\textbf{1}_{3.4\leq x\leq 4}(x-3.4)^2(4-x)^2$.}
\end{figure}

\begin{table}[h]
\begin{center}
\begin{tabular}{|c|c|c|c|c|} \hline  $A$ & $0.1$ & $0.2$ &  $0.3$ & $0.5$ \\ \hline $\Vert h-h_{\text{app}}\Vert_{\text{L}^2(\R)}/\Vert h\Vert_{\text{L}^2(\R)}$ & $8.82\%$ & $10.41\%$  &$15.12\%$ & $54.99\%$\\ \hline \end{tabular} \vspace{5mm}
\end{center}
\caption{\label{tableau2} Relative errors on the reconstruction of $h$ for different amplitudes $A$. We choose $h(x)=A\textbf{1}_{3\leq x\leq 5}(x-3)^2(5-x)^2$ and $g(x)=0$. In every reconstruction, $K$ is the discretization of $[0.01,40]\setminus \{[n\pi-0.2,n\pi+0.2], n\in \N\}$ with $100$ points, $X$ is the discretization of $[1,7]$ with $601$ points and we use the algorithm \eqref{algo} with $\lambda=0.08$ to reconstruct $h'$.}
\end{table}

\subsection{Detection of inhomogeneities}

Using the method described in section 4.2, we generate the solutions of \eqref{eqinhomo} for a set of frequencies $K$ and we evaluate the solutions on the section $\{1\}\times (0,1)$. As explained in section 3.4 and equation \eqref{gammainhomo}, the data only determines $\Gamma(h_n)(k+k_n)$ for every $k\in K$, $k>n\pi$ where $h_n$ is the $n$-th mode of $h$, and $h$ is the inhomogeneity. We define a number of modes $N$ used for the recontruction of $h$, and with the algorithm \eqref{algo}, we find an approximation of $h_n$ for every $n\leq N$. In Figure \ref{testinhomopetit}, we show the reconstruction of $h_n$ for $0\leq n\leq N=9$. We obtain an approximation of $h$ by using the expression $h(x,y)=\sum_{n\in \N} h_n(x)\varphi_n(y)$. Figures \ref{testinhomo1} and \ref{testinhomo2} show two reconstructions of $h$. In the first one, $h$ has a small support and is very well reconstructed. In the latter, the support of $h$ is larger. And albeit it does not yield a good approximation of $h$, it allows localization of the inhomogeneity in the waveguide. Moreover, if we assume that $h$ is a positive function, we can improve the algorithm \eqref{algo} by reconstructing $h$ on each step and projecting on the space of positive functions (see the third part of Figure \ref{testinhomo2}). 

\begin{figure}[!h]
\centering
\resizebox{4.8in}{!}{\input{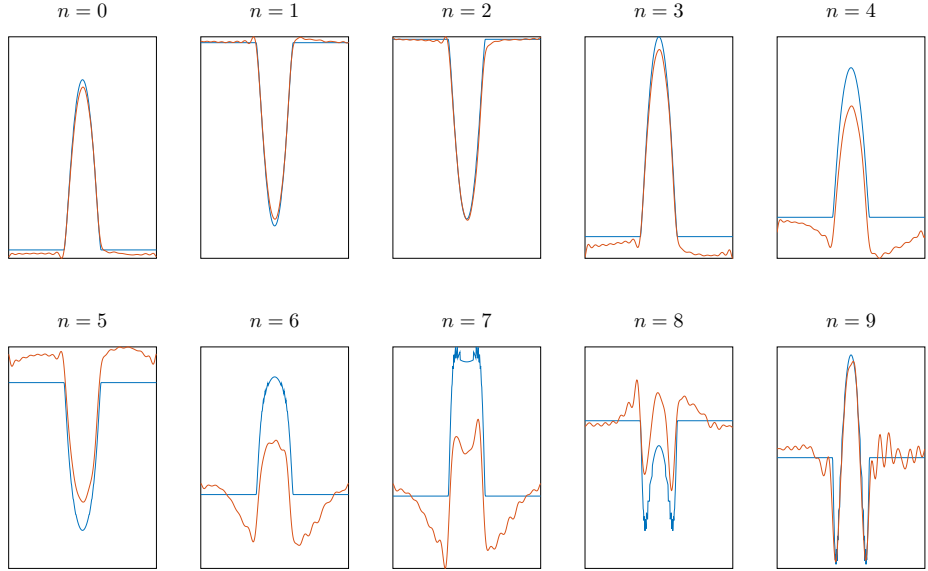}}
\caption{\label{testinhomopetit} Recontruction of $h_n$ for $0\leq n\leq 9$, where $h(x)=0.05\textbf{1}_{\left|\left(\frac{x-4}{0.05},\frac{y-0.6}{0.15}\right)\right|\leq 1}\left|\left(\frac{x-4}{0.05},\frac{y-0.6}{0.15}\right)\right|^2$. In blue, we represent $h_n$ and in red the reconstruction of $h_{n_{\text{app}}}$. In every reconstruction, $K$ is the discretization of $[0.01,150]$ with $200$ points, $X$ is the discretization of $[3.8,4.2]$ with $101$ points and we use the algorithm \eqref{algo} with $\lambda=0.002$ to reconstruct every $h_n$.}
\end{figure}

\begin{figure}[!h]
\centering
\resizebox{4.8in}{!}{
\begin{tikzpicture}[scale=1]
\begin{axis}[width=5cm, height=3cm, axis on top, scale only axis, xmin=3.8, xmax=4.2, ymin=0, ymax=1, colorbar,point meta min=-0.005,point meta max=0.05, title={$h$}]
\addplot graphics [xmin=3.8,xmax=4.2,ymin=0,ymax=1]{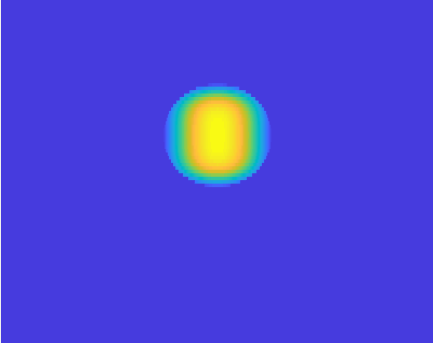};
\end{axis}
\end{tikzpicture}\begin{tikzpicture}[scale=1]
\begin{axis}[width=5cm, height=3cm, axis on top, scale only axis, xmin=3.8, xmax=4.2, ymin=0, ymax=1, colorbar,point meta min=-0.005,point meta max=0.05, title={$h_{\text{app}}$}]
\addplot graphics [xmin=3.8,xmax=4.2,ymin=0,ymax=1]{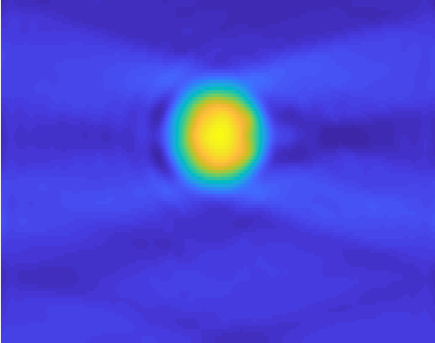};
\end{axis}
\end{tikzpicture}}
\caption{\label{testinhomo1} Recontruction of an inhomogeneity $h$, where $h(x)=0.05\textbf{1}_{\left|\left(\frac{x-4}{0.05},\frac{y-0.6}{0.15}\right)\right|\leq 1}\left|\left(\frac{x-4}{0.05},\frac{y-0.6}{0.15}\right)\right|^2$. On the left, we represent the initial shape of $h$, and on the right the reconstruction $h_{\text{app}}$. Here, $K$ is the discretization of $[0.01,150]$ with $200$ points, $X$ is the discretization of $[3.8,4.2]$ with $101$ points and we use the algorithm \eqref{algo} with $\lambda=0.002$ to reconstruct every $h_n$. We used $N=20$ modes to reconstruct $h$.}
\end{figure}

\begin{figure}[h!]
\centering
\resizebox{4.8in}{!}{
\begin{tikzpicture}[scale=1]
\begin{axis}[width=10cm, height=3cm, axis on top, scale only axis, xmin=3, xmax=6, ymin=0, ymax=1, colorbar,point meta min=-0.015,point meta max=0.07, title={$h$}]
\addplot graphics [xmin=3,xmax=6,ymin=0,ymax=1]{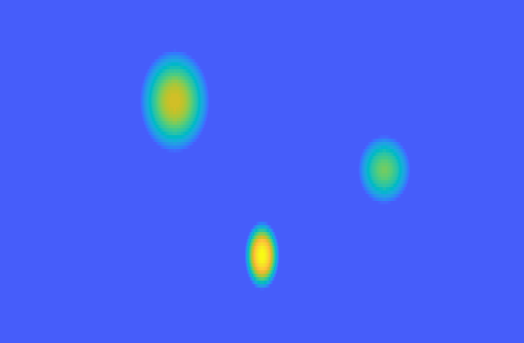};
\end{axis}
\end{tikzpicture}}

\resizebox{4.8in}{!}{\begin{tikzpicture}[scale=1]
\begin{axis}[width=10cm, height=3cm, axis on top, scale only axis, xmin=3, xmax=6, ymin=0, ymax=1, colorbar,point meta min=-0.015,point meta max=0.03, title={$h_{\text{app}}$}]
\addplot graphics [xmin=3,xmax=6,ymin=0,ymax=1]{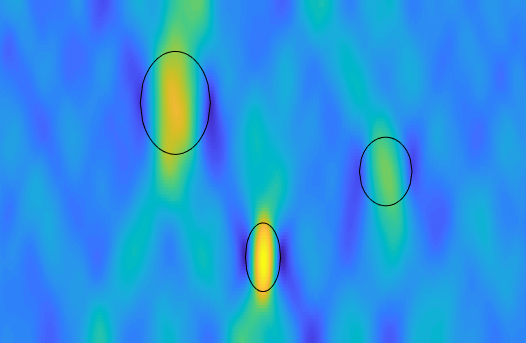};
\end{axis}
\end{tikzpicture}}

\resizebox{4.8in}{!}{\begin{tikzpicture}[scale=1]
\begin{axis}[width=10cm, height=3cm, axis on top, scale only axis, xmin=3, xmax=6, ymin=0, ymax=1, colorbar,point meta min=-0.015,point meta max=0.04, title={$h_{\text{app}}$, $h_{\text{app}}\geq 0$}]
\addplot graphics [xmin=3,xmax=6,ymin=0,ymax=1]{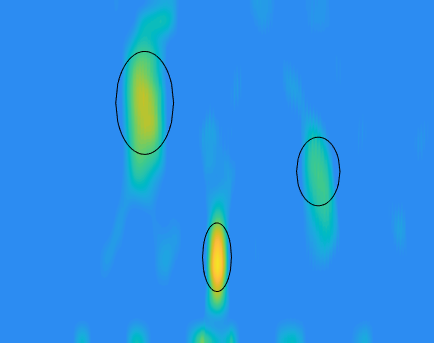};
\end{axis}
\end{tikzpicture}}
\caption{\label{testinhomo2} Recontruction of an inhomogeneity $h$. From top to bottom, the initial representation of $h$, the reconstruction $h_{\text{app}}$ and the reconstruction $h_{\text{app}}$ with the knowledge of the positivity of $h$. Here, $K$ is the discretization of $[0.01,150]$ with $200$ points, $X$ is the discretization of $[3,6]$ with $3001$ points and we use the algorithm \eqref{algo} with $\lambda=0.01$ to reconstruct every $h_n$. We choose used $N=20$ modes to reconstruct $h$.}
\end{figure}

\section{Conclusion}

In this paper, we present a new approach to recover defects in a waveguide. By sending the first propagative mode for frequencies in a given interval, the scattered wave field generated by the defects are measured on a slice of the waveguide. Based on the Fourier transform and the Born approximation, we propose a method to reconstruct the parameters of the defect. We provide a control of the error in the approximation of the parameters of the defects if they are ‘‘small’’ enough so that the Born approximation makes sense.

Our numerical results show that the method works well for the three types of defects considered : bends, bumps, localized inhomogeneities. From measurements generated by a finite element method, we were able to numerically recover the different types of defects using the modal decomposition and a penalized least square algorithm. Our reconstruction of inhomogeneities is similar to the one presented in \cite{dediu1}. While the number of propagative modes sent in the waveguide can be increased, for the method presented in this work, to improve the reconstruction, so can we increase the number of frequencies. 

Our work could be extended to other types of defects such as impenetrable obstacles or cracks in the waveguide. One could also try to apply this multi-frequency point of view to elastic waveguides, where a modal decomposition in terms of Lamb waves is also available. 

\section*{Appendix A: Proof of Proposition \ref{source} and \ref{bord}}
We begin with the proof of Proposition \ref{source}. Let $r>0$, $\Omega_r=(-r,r)\times (0,1)$ and $T$ the application defined by
\begin{equation*}T:\begin{array}{rcl} 
\text{L}^2(\Omega_r) &\rightarrow & \text{L}^2(-r,r)^\N \\
u & \mapsto & \left(\displaystyle\int_0^1 u(x,y) \varphi_n(y) \dd y \right)_{n\in \N}\end{array} .\end{equation*}
Let $\text{H}_1$ denote the Hilbert space 
\begin{equation*} \text{H}_1:=\left\{(u_n)\in \ell^2(\text{H}^1(-r,r)), \sum_{n\in \N} n^2 \Vert u_n \Vert_{\text{L}^2(-r,r)}^2<+\infty\right\},\end{equation*}
equipped with the inner product 
\begin{equation*}\langle (u_n),(v_n)\rangle_{\text{H}_1}=\sum_{n\in \N}(1+n^2\pi^2)\langle u_n,v_n\rangle_{\text{L}^2(-r,r)}+\sum_{n\in \N}\langle u_n',v_n'\rangle_{\text{L}^2(-r,r)}.\end{equation*}
The mapping $T$ is a Hilbert isomorphism between $\text{H}^1(\Omega_r)$ and $\text{H}_1$. 

The variation formulation of \eqref{directsource} takes the form, for every $v\in \text{H}^1(\Omega_r)$, 
\begin{equation} \label{varform}
\int_{\Omega_r} \nabla u_k \nabla v-k^2\int_{\Omega_r}u_kv=\int_{\Omega_r} sv.
\end{equation}
Set $T(u_k)=(u_{k,n})_{n\in \N}$, $T(v)=(v_n)_{n\in \N}$, and notice that the above variational formulation is equivalent to the sequence of problems
\begin{equation} \label{green}
\forall v_n\in \text{H}^1(-r,r), \qquad \int_{-r}^r u_{k,n}'v_n'+(n^2\pi^2-k^2)\int_{-r}^r u_{k,n}v_n=\int_{-r}^r s_nv_n.
\end{equation}
Setting $k_n^2=k^2-n^2\pi^2$ with $\text{Re}(k_n)>0$, $\text{Im}(k_n)>0$, the formulation \eqref{green} is associated to the equation $u_{k,n}''+k_n^2 u_{k,n}=s_n$. We notice that $G_{k_n}(x)=\frac{i}{2k_n}e^{ik_n|x|}$ satisfies the equation $G_{k_n}''+k_n^2G_{k_n}=-\delta_0$ and so we define $u_k:=T^{-1}\left((G_{k_n}\ast s_n)_{n\in \N}\right)$ and note that $u_k\in \text{H}^1(\Omega_r)$ and satisfies \eqref{varform}. Moreover, $u_k$ is outgoing. Finally, using results from elliptic regularity theory (see \cite{grisvard1}) we deduce that $u_k\in \text{H}^2(\Omega_r)$. As this result holds for every $r>0$, we conclude that $u_k\in \text{H}^2\loc(\Omega)$.

The solution of \eqref{directbord} is constructed by the same method. The variational formulation gives for every 
$v\in \text{H}^1(\Omega_r)$, 
\begin{equation} \label{varform2}
\int_{\Omega_r} \nabla u_k \nabla v-k^2\int_{\Omega_r}u_kv-\int_{\Omega_{\text{top}}}b_1v-\int_{\Omega_{\text{bot}}}b_2v=0. \end{equation}
This formulation is equivalent to 
\begin{equation*}
\int_{-r}^r u_{k,n}'v_n'+(n^2\pi^2-k^2)\int_{-r}^r u_{k,n}v_n=\int_{-r}^r (b_1\varphi_n(1)+b_2\varphi_n(0))v_n \quad \forall n\in \N.
\end{equation*}
The function $u_k:=T^{-1}\left((G_{k_n}\ast (b_1\varphi_n(1)+b_2\varphi_n(0)))_{n\in\N}\right)$ is in $\text{H}^1(\Omega_r)$, is outgoing and satisfies \eqref{varform2}. From the elliptic regularity theory, we deduce that $u_k\in \text{H}^2(\Omega_r)$. 

To prove uniqueness for both problems, we notice that $u_k$ satisfies $\Delta u_k+k^2u_k=0$ in $\Omega$. Thus, $u_k$ is a classical solution and $u_k\in \mathcal{C}^\infty(\Omega)$ and can be written as a linear combination of $(x,y)\mapsto \varphi_n(y)e^{\pm ik_n x}$. The outgoing caracter of $u_k$ shows that $u_k=0$ if $s=0$ or $b_1=b_2=0$.

\section*{Appendix B: Proof of Proposition \ref{directh} and \ref{directg}}
We begin with the proof of proposition \ref{directh}. Using the same notation as in Appendix A, the function $G_{k_n}$ satisfies 
\begin{equation*}
\Vert G_{k_n}\Vert_{\text{L}^1(-2r,2r)}\leq\left\{\begin{array}{cl}\frac{2r}{k_n} & \text{ if } n<k/\pi, \\ \frac{1}{|k_n|}\min\left(\frac{1}{|k_n|},2r\right) & \text{ if } n>k/\pi, \end{array}\right. \end{equation*}
\begin{equation*} \Vert G_{k_n}'\Vert_{\text{L}^1(-2r,2r)}\leq\left\{\begin{array}{cl}2r & \text{ if } n<k/\pi, \\ \min\left(\frac{1}{|k_n|}, 2r\right) & \text{ if } n>k/\pi. \end{array}\right. \end{equation*}

We define $\delta=\min_{n\in \N}\left(\sqrt{|k^2-n^2\pi^2|}\right)$, and apply the Young inequality to $u_{k,n}$:
\begin{equation*}
\Vert u_{k,n} \Vert_{\text{L}^2(-r,r)}\leq \Vert G_{k_n}\Vert_{\text{L}^1(-2r,2r)}\Vert s_n \Vert_{\text{L}^2(-r,r)}.
\end{equation*}
This leads to 
\begin{equation*} \Vert u_k \Vert_{\text{L}^2(\Omega_r)}^2\leq \frac{4r^2}{\delta^2} \sum_{n\in\N}\Vert s_n\Vert_{\text{L}^2(-r,r)}^2=\frac{4r^2}{\delta^2}\Vert s\Vert_{\text{L}^2(\Omega_r)}^2.\end{equation*}
Applying the Young inequality to $u_{k,n}'$, we get
\begin{equation*}\begin{split} \Vert \nabla u_k \Vert_{\text{L}^2(\Omega_r)}^2\leq 4r^2\sum_{n<k/\pi}\left(1+\frac{n^2\pi^2}{k_n^2}\right)\Vert s_n\Vert_{\text{L}^2(-r,r)}^2\hspace{3cm} \\ \hspace{3cm}+4r^2\sum_{n>k/\pi}\left(1+\frac{n^2\pi^2}{|k_n|^2}\right)\Vert s_n\Vert^2_{\text{L}^2(-r,r)}.\end{split}\end{equation*}
If $N$ is the largest propagative mode and if $n>k/\pi>N$, 
\begin{equation*} 1+\frac{n^2\pi^2}{|k_n|^2}\leq 1+\frac{(N+1)^2\pi^2}{|k_{N+1}|^2}\leq 1+\frac{(N+1)^2\pi^2}{\delta^2},\end{equation*}
so that
\begin{equation*} \Vert \nabla u_k \Vert_{\text{L}^2(\Omega_r)}^2\leq 4r^2\left(1+\frac{(k+\pi)^2}{\delta^2}\right)\Vert s\Vert_{\text{L}^2(\Omega_r)}^2. \end{equation*}
Finally, we notice that 
\begin{equation*}\begin{split}
\Vert \nabla^2 u_k \Vert_{\text{L}^2(\Omega_r)}^2=\sum_{n\in \N} n^4\pi^4\Vert u_{k,n} \Vert^2_{\text{L}^2(-r,r)}+\sum_{n\in \N} 2n^2\pi^2\Vert u_{k,n}'\Vert^2_{\text{L}^2(-r,r)} \qquad \qquad \\ \qquad +\sum_{n\in \N} \Vert u_{k,n}''\Vert_{\text{L}^2(-r,r)}^2,\end{split}
\end{equation*}
and that 
\begin{equation*} \Vert u_{k,n}''\Vert_{\text{L}^2(-r,r)}^2=\Vert -s_n-k_n^2u_{k,n}\Vert_{\text{L}^2(-r,r)}^2\leq \left(\Vert s_n \Vert_{\text{L}^2(-r,r)}+k_n^2\Vert u_{k,n}\Vert_{\text{L}^2(-r,r)}\right)^2.\end{equation*}
Combining both relations yields 
\begin{equation*}\begin{split}\Vert \nabla^2 u \Vert_{\text{L}^2(\Omega_r)}^2\leq \sum_{n<k/\pi}\left(\frac{4r^2n^4\pi^4}{k_n^2}+8n^2\pi^2r^2+(1+2k_nr)^2\right)\Vert s_n\Vert_{\text{L}^2(-r,r)}^2 \qquad \qquad \\ \qquad \qquad 
+\sum_{n>k/\pi}\left[\frac{n^4\pi^4}{|k_n|^4}+\frac{2n^2\pi^2}{|k_n|^2}+\left(1+\frac{|k_n|^2}{|k_n|}\min\left(\frac{1}{|k_n|},2r\right)\right)^2\right]\Vert s_n\Vert^2_{\text{L}^2(-r,r)}.\end{split}\end{equation*}
If $N$ is the largest propagative mode, and if $n>k/\pi$, 
\begin{equation*}
\frac{n^4\pi^4\min\left(\frac{1}{|k_n|^2},4r^2\right)}{|k_n|^2}+2n^2\pi^2\min\left(\frac{1}{|k_n|^2},4r^2\right)\leq \frac{(k+\pi)^4}{\delta^4}+2\frac{(k+\pi)^2}{\delta^2},\end{equation*}
and the following estimate holds
\begin{equation*} \begin{split} \Vert \nabla^2 u \Vert_{\text{L}^2(\Omega_r)}^2
\leq [\,\max\left(4r^2,\frac{1}{\delta^2}\right)\left(\frac{(k+\pi)^4}{\delta^2}+2(k+\pi)^2\right)\hspace{3cm}\\ \hspace{1cm}+\max((1+2kr)^2,4)\,]\, \Vert s \Vert_{\text{L}^2(\Omega_r)}^2.\end{split} \end{equation*}

To prove Proposition \ref{directg}, we deduce from the results in \cite{grisvard1} that there exists constants $d(r)$ and $\mu>0$, such that 
\begin{equation*} \Vert u_k \Vert_{\text{H}^2(\Omega_r)}\leq d(r)\left(\Vert -\Delta u_k +\mu u_k \Vert_{\text{L}^2(\Omega_r)}+\Vert b_1 \Vert_{\text{H}^{1/2}(-r,r)}+\Vert b_2 \Vert_{\text{H}^{1/2}(-r,r)}\right), \end{equation*}
and if follows that
\begin{equation*}\Vert u \Vert_{\text{H}^2(\Omega_r)} \leq d(r) \left((k^2+\mu) \Vert u \Vert_{\text{L}^2(\Omega_r)}+\Vert b_1 \Vert_{\text{H}^{1/2}(-r,r)}+\Vert b_2 \Vert_{\text{H}^{1/2}(-r,r)}\right).\end{equation*}
Using the same method as for the estimation of $\Vert u_{k,n}\Vert_{\text{L}^2(-r,r)}$ with the Young inequality, we get
\begin{equation*}  \Vert u \Vert_{\text{L}^2(\Omega_r)}^2\leq \left(\Vert b_1 \Vert^2_{\text{H}^{1/2}(-r,r)}+\Vert b_2 \Vert^2_{\text{H}^2{1/2}(-r,r)}\right) \left( \sum_{n<k/\pi} \frac{4r^2}{k_n^2}+\sum_{n>k/\pi} \frac{1}{k_n^4}\right).\end{equation*}
Finally, we obtain
\begin{equation*} D=d(r)\left((k^2+\mu)\max\left(2r,\frac{1}{\delta}\right)\sqrt{\sum_{n\in \N} \frac{1}{k_n^2}}+1\right).\end{equation*}


\medskip
Received xxxx 20xx; revised xxxx 20xx.
\medskip

\end{document}